\documentclass[12pt,oneside]{amsart}
\title{On Lipschitz maps and the H\"older regularity of flows}
\date{}
   \author{Youness Boutaib}
   \thanks{Part of this research was carried while at the University of Oxford with the support of an Oxford-Man Institute for Quantitative Finance scholarship, for which the author is grateful. The author is also thankful for the support of the DFG through the research unit FOR2402 in Berlin and Potsdam where significant improvements on this work have been done.}
   \address{Institut f\"ur Mathematik, Technische Universit\"at Berlin, Str{\ss}e des 17. Juni 136, 10623 Berlin, Germany.}
   \address{Institut f\"ur Mathematik, Universit\"at Potsdam, Haus 9, Karl-Liebknecht-Stra{\ss}e 24, 14476 Potsdam, Germany. }
   \email{boutaib.youness@gmail.com}

\usepackage{amssymb} % necessaire pour les mathbb 
\usepackage{amsfonts}
\usepackage{amsmath} %pour underset
\usepackage{shuffle} %pour le shuffle product notation
\usepackage{amsthm} %pour des theoremes unnumbered
\usepackage[margin=1in]{geometry}  % set the margins to 1in on all sides
\newtheorem{theo}{Theorem}[section]
\newtheorem{lemma}[theo]{Lemma}
\newtheorem*{example}{Example}
\newtheorem{Cor}[theo]{Corollary}
\newtheorem{Def}[theo]{Definition}
\newtheorem{Prop}[theo]{Proposition}
\newtheorem{Rem}[theo]{Remark}
\newtheorem*{Claim}{Claim}
\newtheorem*{Notation}{Notation}
\newtheorem*{examples}{Examples}
\begin{document}
\maketitle

%%%%
% ABSTRACT
%%%%
	\begin{abstract} This paper regroups some of the basic properties of Lipschitz maps and their flows. Many of the results presented here are classical in the case of smooth maps. We prove them here in the Lipschitz case for a better understanding of the Lipschitz geometry and for a quantification of the related properties, which would be of use to the development of numerical methods for rough paths for example. We also introduce the notion of almost Lipschitz maps, which provide a sharper control and description of flows of Lipschitz vector fields and local inverses of Lipschitz injective immersions.\end{abstract} 

%%%
%NOTATION
%%%%
\section*{Basic notations}
\begin{tabular}{lll}
$\mathcal{S}_k$&:& The symmetric group of order $k$.\\
$\lfloor\gamma\rfloor$&:& The only integer such that $0<\gamma- \lfloor\gamma\rfloor\leq 1$, $\gamma$ being a real number.\\
$[\gamma]$&:& The integer part of a real number $\gamma$, i.e. the only integer such that\\
&& $0\leq\gamma-[\gamma]<1$.\\
$\mathcal{L}_c(E,F)$&:& The space of all continuous linear mappings from a normed vector \\
		&&space $E$ to a normed vector space $F$.\\
$E^{\otimes n}$&:& The space of homogenous tensors of the vector space $E$ of order $n$, \\
		&&$n\in\mathbb{N}^*$.\\		
$\mathcal{L}_s(E^{\otimes k},F)$&:& The space of symmetric $k$-linear mappings from a vector space $E$  \\
		&& to a vector space $F$.\\
$\mathrm{I}_k$&:& The identity matrix of rank $k$.\\
$\mathrm{Id}_U$&:& The identity map on the set $U$.\\
$\overline{A}$&:& The closure of a subset $A$ of a topological space.\\
$B(x,\alpha)$&:& The ball centered at $x$ of radius $\alpha$.\\
\end{tabular}

%%%%
%%%
%%Lipschitz maps
%%%
%%%%
\section{Lipschitz maps}\label{sec:lip-maps}
	In \cite{Young}, L.C. Young uses the concept of $p$-variation $(p\geq1)$ as a means to characterise the smoothness of a path to generalize Stieltjes' integration theory to paths of finite $p$-variation. In building a theory of differential equations ($\mathrm{d}y=f(y)\mathrm{d}x$) using the aformentioned work, one needs to be able to control the smoothness (in terms of variation) of the image of a path of finite $p$-variation under the involved vector fields (i.e. $f(y)$). It appears that Lipschitz maps, introduced by Whitney (\cite{Whitney2}) and studied for example by Stein in \cite{Stein}, are the appropriate type of maps to use in this framework  and the wider one of rough paths, introduced by Lyons in \cite{Revista}. Lipschitz maps have the advantage (among others) of making sense even on discrete sets. They correspond to a variation of the H\"older modelled distributions in the classical polynomial regularity structure (more precisely, a H\"older modelled distribution in this regularity structure is a Lipschitz map on every compact set. For more details, see Hairer \cite{Hairer}). Regularity structures are used to solve a large class of sub-critical stochastic parabolic PDEs. The present work may constitute a simplified example (i.e. involving a ladder structure) of the type of algebraic constructions that one needs to build a closed regularity structure that can be used in solving more elaborate SPDEs. In this section, we give the definition of Lipschitz maps then set out to answer basic and natural questions about this class of maps: do they have a nice embedding structure? How can they be linked to the more familiar class of $\mathcal{C}^n$ maps? Are they stable under compostion? etc.

	%%%%%%
	%Basic definitions and generalities
	%%%%%%
	%%%%%%
	\subsection{Basic definitions and properties, norms on tensor product spaces}\label{SectionDefinitionNorms}
	We first fix our definition of Lipschitz maps. To represent multi-linear maps such as higher derivatives, we opt for a representation by linear maps taking values in tensor product spaces as can be found for example in \cite{Flour}:
	%%%%	
	%%%% DEFINITION LIPSCHITZ FUNCTIONS
	%%%%
	\begin{Def}\label{DefLipMap} Let $n\in \mathbb{N}$ and $0< \varepsilon\leq 1$. Let $E$ and $F$ be two normed vector spaces and $U$ be a subset of $E$. Let $f^0:U\to E$ be a map. For every $k \in [\![1,n]\!]$, let $f^k:U\to \mathcal{L}_s(E^{\otimes k},F)$ be a map with values in the space of the symmetric $k$-linear mappings from $E$ to $F$. We will use, without ambiguity, the same notation $\|.\|$ to designate norms on $E^{\otimes k}$, for $k\in [\![1,n]\!]$, and the norm on $F$. \\
	For $k \in [\![0,n]\!]$, the map $R_k:E\times E\to \mathcal{L}(E^{\otimes k},F)$ defined by:
	\[\forall x,y \in U, \forall v \in E^{\otimes k}: f^k(x)(v)=\sum\limits_{j=k}^{n} f^j(y)(\frac{v\otimes (x-y)^{\otimes (j-k)}}{(j-k)!})+R_k(x,y)(v)\]
	is called the remainder of order $k$ associated to $f=(f^0,f^1,\ldots,f^n)$.\\	
	The collection $f=(f^0,f^1,\ldots,f^n)$ is said to be Lipschitz of degree $n+\varepsilon$ on $U$ (or in short a $\textrm{Lip}-(n+\varepsilon)$ map) if there exists a constant $M$ such that for all $k \in [\![0,n]\!]$, $x,y \in U$ and $v_1,\ldots,v_k \in E$:
	\begin{enumerate}
	\item $\|f^k(x)(v_1\otimes\cdots\otimes v_k)\| \leq M \|v_1\otimes\cdots\otimes v_k\| $;
	\item $\|R_k(x,y)(v_1\otimes\cdots\otimes v_k)\|\leq M \|x-y\|^{n+\varepsilon-k}\|v_1\otimes\cdots\otimes v_k\|$.
	\end{enumerate}
     The smallest constant $M$ for which the properties above hold is called the $\textrm{Lip}-(n+\varepsilon)$-norm of $f$ and is denoted by $\|f\|_{\textrm{Lip}-(n+\varepsilon)}$.
     \end{Def}
	Note that the above definition is purely quantitative and does not require any properties (in particular topological ones) from the domain of the definition of the Lipschitz map.
	%%%%
	%%%% Derivatives of a Lip map
	%%%%
	\begin{Rem}
	On any non-empty open subset of $U$ (and in particular on the interior of $U$), $f^1,\ldots,f^n$ are the successive derivatives of $f^0$. However, these maps are not necessarily uniquely determined by $f^0$ on an arbitrary set $U$. Keeping this in mind, if $f^0:U\to F$ is a map such that there exist $f^1,\ldots,f^n$ such that $(f^0,f^1,\ldots,f^n)$ is $\textrm{Lip}-(n+\varepsilon)$, we will often say that $f^0$ is $\textrm{Lip}-(n+\varepsilon)$ with no mention of $f^1,\ldots,f^n$.
     \end{Rem}
    %%%%
	%%%% Invariance of Lipschitzness under change of norm
	%%%%
     \begin{Rem} It is clear that the property of being Lipschitz is invariant under the change of norms by equivalent ones.
     \end{Rem}
	One important feature one has to pay attention to when deriving properties of Lipschitz maps is the nature of the norms on the tensor spaces. We study here three types of norms which will be of use in the exposition of our work.
	%%%%
	%%%% Projective norms
	%%%%	
	\begin{Def}[Projective property] Let $E$ be a normed vector space. Let $n\in \mathbb{N}^*$. We say that $(E^{\otimes k})_{1\leq k \leq n}$ (respectively $(E^{\otimes k})_{k \geq 1}$) are endowed with norms satisfying the projective property if, for every $k\in [\![1,n]\!]$ (resp. $k\geq1$) and $p,q \in \mathbb{N}$ such that $p+q=k$ and every $a\in E^{\otimes p}, b\in E^{\otimes q}$, we have $\|a\otimes b\|\leq\|a\|\|b\|$.
	\end{Def}
	When at least such a family of norms exists, norms satisfying the projective property are abundant in the following sense:
	\begin{Prop}\label{ChangeProjectiveNorm}  Let $E$ be a normed vector space and $n\in \mathbb{N}^*$. Suppose $(\|.\|_k)_{1\leq k \leq n}$ are norms on $(E^{\otimes k})_{1\leq k \leq n}$ satisfying the projective property, then, for $\alpha>0$ and $\beta\geq1$, the norms $(\alpha^k\|.\|_k)_{1\leq k \leq n}$ and $(\beta\|.\|_k)_{1\leq k \leq n}$ also satisfy the projective property.
	\end{Prop}
	%%%% Example of Projective norms %%%%
	\begin{example} Let $E$ be a finite dimensional vector space and let $(\vec{e}_1,\ldots,\vec{e}_r)$ be a basis for $E$. Let $n\in\mathbb{N}^*$. Let $k\in[\![1,n]\!]$ and $p\geq1$ and define the norms $\|.\|_{p,k}$ and $\|.\|_{\infty,k}$ on $E^{\otimes k}$ by the following: for $x\in E^{\otimes k}$, if $(x_{i_1,\ldots,i_k})_{1\leq i_1,\ldots,i_k\leq r}$ are  the coordinates of $x$ in the basis $(\vec{e}_{i_1}\otimes\cdots\otimes \vec{e}_{i_k})_{1\leq i_1,\ldots,i_k\leq r}$ of $E^{\otimes k}$, i.e.:
	\[x=\sum\limits_{1\leq i_1,\ldots,i_k\leq r}x_{i_1,\ldots,i_k}\vec{e}_{i_1}\otimes\cdots\otimes \vec{e}_{i_k}\]
	then:
	\[\|x\|_{p,k}= \left(\sum\limits_{1\leq i_1,\ldots,i_k\leq r}|x_{i_1,\ldots,i_k}|^p\right)^{1/p}\quad\textrm{and}\quad\|x\|_{\infty,k}=\max_{1\leq i_1,\ldots,i_k\leq r} |x_{i_1,\ldots,i_k}|\]
	Then $(\|.\|_{p,k})_{1\leq k \leq n}$ and $(\|.\|_{\infty,k})_{1\leq k \leq n}$ are norms on $(E^{\otimes k})_{1\leq k \leq n}$ satisfying the projective property.\end{example}
	%%%% Linear maps on product space %%%%
	\begin{Def} Let $E$ and $F$ be two normed vector spaces and $u:E\to F$ be a linear map. Let $n\in\mathbb{N}^*$. We define the map $u^{\otimes n}:E^{\otimes n}\to F^{\otimes n}$ as the unique linear map satisfying:
	\[\forall v_1,\ldots,v_n\in E:\quad u^{\otimes n}(v_1\otimes\cdots\otimes v_n)=u(v_1)\otimes\cdots\otimes u(v_n)\]
	\end{Def}
	\begin{Rem} The existence of such a map is a consequence of the universal property defining tensor product spaces.
	\end{Rem}
	%%%% Compatible norms %%%%
	\begin{Def}[Compatible norms] Let $E$ and $F$ be two normed vector spaces. Let $n\in\mathbb{N}^*$ and $C\geq0$. We say that $(E^{\otimes k})_{1 \leq k\leq n}$ and $(F^{\otimes k})_{1 \leq k\leq n}$ are endowed with $C$-compatible norms if, for every bounded linear map $u:E\to F$ and every $k\in [\![1,n]\!]$, we have $\|u^{\otimes k}\|\leq C\|u\|^k$. When the value of $C$ is irrelevant, we may simply say that the norms are compatible and assume that $C=1$.
	\end{Def}
	%%%% Example of Compatible norms %%%%
	\begin{examples} Let $E$ be a finite dimensional vector space and let $(\vec{e}_1,\ldots,\vec{e}_r)$ be a basis for $E$. Let $n\in\mathbb{N}^*$. Let $F$ be a normed vector space. We assume that we have norms on $(F^{\otimes k})_{1\leq k \leq n}$ satisfying the projective property. Then:\begin{itemize}
	\item Let $p\geq1$. The norms $(\|.\|_{p,k})_{1\leq k \leq n}$ (resp. $(r^{k(1-1/p)}\|.\|_{p,k})_{1\leq k \leq n}$) on $(E^{\otimes k})_{1\leq k \leq n}$ are $r^{n(1-1/p)}$-compatible (resp. $1$-compatible) with the norms on $(F^{\otimes k})_{1\leq k \leq n}$.
	\item The norms $(\|.\|_{\infty,k})_{1\leq k \leq n}$ (resp. $(r^k\|.\|_{\infty,k})_{1\leq k \leq n}$) on $(E^{\otimes k})_{1\leq k \leq n}$ are $r^n$-compatible (resp. $1$-compatible) with the norms on $(F^{\otimes k})_{1\leq k \leq n}$.
	\end{itemize}	\end{examples}
	\begin{Rem} As shown in the case of the norms given in the previous examples and by proposition \ref{ChangeProjectiveNorm}, if the norms on $(E^{\otimes k})_{1 \leq k\leq n}$ and $(F^{\otimes k})_{1 \leq k\leq n}$ (or $(E^{\otimes k})_{k\geq 1}$ and $(F^{\otimes k})_{k\geq 1}$) are $C$-compatible, it is always possible to define new norms that are equivalent to the original norms so that the new norms are $1$-compatible and that the new norms on $(E^{\otimes k})_{1\leq k \leq n}$ satisfy the projective property if the original ones do.
	\end{Rem}
	%%%% Permutation of tensors %%%%		
	\begin{Def}[Action of the Symmetric Group on Tensors]
	Let $n \in \mathbb{N}^*$, $\sigma \in \mathcal{S}_n$ and $E$ be a vector space. We define the action of $\sigma$ on the homogenous tensors of $E$ of order $n$ as a linear map by the following:
	\[\forall x_1,x_2,\ldots, x_n \in E \quad \sigma(x_1\otimes x_2\otimes \cdots \otimes x_n)=x_{\sigma(1)}\otimes x_{\sigma(2)}\otimes \cdots \otimes x_{\sigma(n)}\]
	\end{Def}
	\begin{Notation} In the context of the previous definition, for $\sigma\in\mathcal{S}_{n}$, $i\in[\![0,n]\!]$ and $x_1,x_2,\ldots, x_n \in E$, we define $\sigma^{1,i}(x)$ ($x:=x_1\otimes x_2\otimes\cdots\otimes x_n$) and $\sigma^{2,i}(x)$ to be the only elements of $E^{\otimes i}$ and $E^{\otimes (n-i)}$ respectively such that:
	\[\sigma(x)=\sigma^{1,i}(x)\otimes\sigma^{2,i}(x)\]
	\end{Notation} 
	%%%% Symmetric norms %%%%
    \begin{Def}[Symmetric norms]
    Let $E$ be a vector space and $n\in\mathbb{N}^*$. A norm on $E^{\otimes n}$ is said to be symmetric if:
	\[\forall n \in \mathbb{N}^*,\,\forall \sigma \in \mathcal{S}_n,\, \forall x \in E^{\otimes n} \quad \| \sigma (x)\|=\|x\|\]
    \end{Def}
	We show now how to control the Lipschitz norm of the Cartesian product of two Lipschitz maps.
	%%%% the product function %%%%
	\begin{Prop}\label{ProductLipFunc} Let $\gamma>0$. Let $E$, $F$ and $G$ be normed vector spaces. Let $U$ be a subset of $E$ and let $f$ (resp. $g$) be a map defined on $U$ with values in $F$ (resp. $G$). Let $h$ be the map defined on $U$ by $h=(f,g)$. Then:\begin{itemize}
	\item If $f$ and $g$ are $\textrm{Lip}-\gamma$ and $F\times G$ is endowed with the $l^p$ norm ($p\in [1,\infty]$), then $h$ is also $\textrm{Lip}-\gamma$ and $\|h\|_{\textrm{Lip}-\gamma}$ is less than or equal to the $l^p$ norm of $(\|f\|_{\textrm{Lip}-\gamma},\|g\|_{\textrm{Lip}-\gamma})$.
	\item If the norm $\|.\|_F$ on $F$ and the norm $\|.\|_{F\times G}$ on $F\times G$ are such that there exists $C>0$ satisfying:
	\[\forall (x,y)\in F\times G:\quad \|x\|_F\leq C \|(x,y)\|_{F\times G} \]
	(note that the $l^p$ norms on $F\times G$ satisfy this property, for $p\in [1,\infty]$), and if $h$ is $\textrm{Lip}-\gamma$ then $f$ is $\textrm{Lip}-\gamma$ and $\|f\|_{\textrm{Lip}-\gamma}\leq C\|h\|_{\textrm{Lip}-\gamma}$.
	\end{itemize}\end{Prop}	
	%%%%%%
	% Local characterization and embeddings
	%%%%%%
	%%%%%%
	\subsection{Local characterization and embeddings}
	Once the concept of Lipschitzness understood, one of the first and the most natural questions one may ask is whether $\textrm{Lip}-\gamma$ maps are $\textrm{Lip}-\gamma'$, for $\gamma\geq\gamma'>0$. We deal first with the trivial case where the domain of definition of the map is bounded:
	%%%% EMBEDDINGS OF LIP BOUNDED SET%%%%
	\begin{lemma}\label{EmbedLip1} Let $\gamma,\gamma'>0$ such that $\gamma'<\gamma$. Let $E$ and $F$ be two normed vector spaces and $U$ be a bounded subset of $E$. Let $f:U\to F$ be a $\textrm{Lip}-\gamma$ map. If $\lfloor\gamma'\rfloor<\lfloor\gamma\rfloor$, we assume that $(E^{\otimes k})_{1\leq k \leq \lfloor\gamma\rfloor}$ are endowed with norms satisfying the projective property. Then $f$ is $\textrm{Lip}-\gamma'$ and if $L\geq0$ is larger than or equal to the diameter of $U$ then:
\[\|f\|_{\textrm{Lip}-\gamma'}\leq\|f\|_{\textrm{Lip}-\gamma}\max\left(1,\sum\limits_{j=\lfloor\gamma'\rfloor+1}^{\lfloor\gamma\rfloor}{\frac{L^{j-\gamma'}}{(j-\lfloor\gamma'\rfloor)!}}+L^{\gamma-\gamma'}\right)\]
	\end{lemma}
	\begin{proof} Let $n,n'\in \mathbb{N}$, $(\varepsilon,\varepsilon')\in (0,1]^2$ such that $\gamma=n+\varepsilon$ and $\gamma'=n'+\varepsilon'$. Let $f^1,\ldots,f^n$ be maps on $U$ such that $(f,f^1,\ldots,f^n)$ is $\textrm{Lip}-\gamma$ and let $R_0,\ldots,R_n$ be the associated remainders. For $k\in[\![0,n']\!]$, define $S_k:U\times U\to\mathcal{L}(E^{\otimes k},F)$ as follows:
	\[\forall x,y \in U, \forall v\in E^{\otimes k}: S_k(x,y)(v)=\sum\limits_{j=n'+1}^{n}{f^j(y)\left(\frac{v\otimes(x-y)^{\otimes (j-k)}}{(j-k)!}\right)}+R_k(x,y)(v)\]
	By a straightforward computation, one gets that, for all $x,y\in U$:
	\[\|S_k(x,y)\|\leq \|f\|_{\textrm{Lip}-\gamma}\left(\sum\limits_{j=n'+1}^{n}{\frac{L^{j-\gamma'}}{(j-k)!}}+L^{\gamma-\gamma'}\right)\|x-y\|^{\gamma'-k}\]
	By recognising the $S_i$'s in the expansion formulas of the $f_i$'s, we see therefore that $(f,f^1,\ldots,f^{n'})$ is $\textrm{Lip}-\gamma'$ with $S_0,\ldots,S_{n'}$ as remainders and:
	\[\|f\|_{\textrm{Lip}-\gamma'}\leq\|f\|_{\textrm{Lip}-\gamma}\max\left(1,\sum\limits_{j=n'+1}^{n}{\frac{L^{j-\gamma'}}{(j-n')!}}+L^{\gamma-\gamma'}\right)\]\end{proof}
	\begin{Rem} With the notations of the previous lemma, if $\lfloor\gamma'\rfloor=\lfloor\gamma\rfloor$, $\sum_{\lfloor\gamma'\rfloor+1}^{\lfloor\gamma\rfloor}{\frac{L^{j-\gamma'}}{(j-\lfloor\gamma'\rfloor)!}}$ is understood to be zero.
	\end{Rem}
	\begin{Rem} With the notations of the previous lemma, we have the following simple control:
	\[\max\left(1,\sum\limits_{j=\lfloor\gamma'\rfloor+1}^{\lfloor\gamma\rfloor}{\frac{L^{j-\gamma'}}{(j-\lfloor\gamma'\rfloor)!}}+L^{\gamma-\gamma'}\right) \leq C_{\gamma,\gamma'}\max\left(1, L^{\gamma-\gamma'}\right)\]
	where $C_{\gamma,\gamma'}=1$ if $\lfloor\gamma'\rfloor=\lfloor\gamma\rfloor$ and $C_{\gamma,\gamma'}=e$ otherwise.
	\end{Rem}
	The aim now is to be able to go from the case where the domain of definition of the map is bounded to a more general one. This gives us an important local characterization of Lipschitz maps:
	%%%% LOCAL CARACTERISATION OF LIP General Set%%%
	\begin{lemma}\label{LocalLipGen} Let $\gamma>0$. Let $E$ and $F$ be two normed vector spaces and $U$ be a subset of $E$. For every $k \in [\![0,\lfloor\gamma\rfloor]\!]$, let $f^k:U\to \mathcal{L}(E^{\otimes k},F)$ be a map with values in the space of the symmetric $k$-linear mappings from $E$ to $F$. We assume that $(E^{\otimes k})_{1\leq k \leq \lfloor\gamma\rfloor}$ are endowed with norms satisfying the projective property and that there exists $\delta>0$ and $C\geq0$ such that, for every $x\in U$, $f_{|B(x,\delta)\cap U}$ is $\textrm{Lip}-\gamma$ with a norm less than or equal to $C$ $($where $f=(f^0,\ldots,f^{\lfloor\gamma\rfloor}))$. Then $f$ is $\textrm{Lip}-\gamma$ and:
	\[\|f\|_{\textrm{Lip}-\gamma}\leq C\max\left(1,\max_{0\leq k\leq \lfloor\gamma\rfloor}{\frac{1}{\delta^{\gamma-k}}(1+\sum\limits_{j=0}^{\lfloor\gamma\rfloor-k} \frac{\delta^j}{j!})}\right)\] 
	\end{lemma}
	\begin{proof} Let $k\in [\![0,\lfloor\gamma\rfloor]\!]$. We already know that $\sup\limits_{x\in U}\|f^k(x)\|\leq C$. Define $R_k:U\times U\to \mathcal{L}(E^{\otimes k},F)$ as follows:
	\[\forall x,y \in E, \forall v\in E^{\otimes k}:\quad R_k(x,y)(v)=f^k(x)(v)-\sum\limits_{j=k}^{\lfloor\gamma\rfloor} f^j(y)(\frac{v\otimes (x-y)^{\otimes (j-k)}}{(j-k)!})\]
	Let $x,y\in U$. If $\|x-y\|<\delta$, then, as $f_{|B(x,\delta)\cap U}$ is $\textrm{Lip}-\gamma$, we have:
	\[\|R_k(x,y)\|\leq C \|x-y\|^{\gamma-k}\]
	Assume that $\|x-y\|\geq\delta$, then, as the $(E^{\otimes k})_{1\leq k \leq n}$ are endowed with norms satisfying the projective property, we obtain:
	\[\begin{array}{rcl}
	\frac{\displaystyle \|R_k(x,y)\|}{\displaystyle \|x-y\|^{\gamma-k}}&\leq& \frac{\displaystyle \|f^k(x)\|}{\displaystyle \|x-y\|^{\gamma-k}}+\sum\limits_{j=k}^{\lfloor\gamma\rfloor} \frac{\displaystyle \|f^j(y)\|}{\displaystyle \|x-y\|^{\gamma-j} (j-k)!}\\
	&\leq&C \left(\frac{\displaystyle 1}{\displaystyle \delta^{\gamma-k}}+\sum\limits_{j=k}^{\lfloor\gamma\rfloor} \frac{\displaystyle 1}{\displaystyle \delta^{\gamma-j}(j-k)!}\right)\\
	&\leq&\frac{\displaystyle C}{\displaystyle \delta^{\gamma-k}}\left(1+\sum\limits_{j=0}^{\lfloor\gamma\rfloor-k} \frac{\displaystyle \delta^j}{\displaystyle j!}\right)\\
	\end{array}\]
	We deduce then that $f$ is $\textrm{Lip}-\gamma$ on $U$ with the suggested upper-bound of $\|f\|_{\textrm{Lip}-\gamma}$.\end{proof}
	\begin{Rem} With the notations of the previous lemma, we have:
	\[\max\left(1,\max_{0\leq k\leq \lfloor\gamma\rfloor}{\frac{1}{\delta^{\gamma-k}}(1+\sum\limits_{j=0}^{\lfloor\gamma\rfloor-k} \frac{\delta^j}{j!})}\right) \leq (1+e) \max\left(1,\frac{1}{\delta^{\gamma}}\right)\]
	\end{Rem}
	We can now state the following natural embedding theorem:
	%%%% EMBEDDINGS OF LIP GENERAL%%%%
	\begin{theo}\label{EmbedLip2} Let $\gamma,\gamma'>0$ such that $\gamma'<\gamma$. Let $E$ and $F$ be two normed vector spaces and $U$ be a subset of $E$. We assume that $(E^{\otimes k})_{1\leq k \leq \lfloor\gamma\rfloor}$ are endowed with norms satisfying the projective property. Let $f:U\to F$ be a $\textrm{Lip}-\gamma$ map. Then $f$ is $\textrm{Lip}-\gamma'$ and there exists a constant $M_{\gamma,\gamma'}$ (depending only on $\gamma$ and $\gamma'$) such that $\|f\|_{\textrm{Lip}-\gamma'}\leq M_{\gamma,\gamma'} \|f\|_{\textrm{Lip}-\gamma}$
	\end{theo}
	\begin{proof} Let $\delta>0$ and $x\in U$. $f$ is $\textrm{Lip}-\gamma$ on $B(x,\delta)\cap U$ with a $\textrm{Lip}-\gamma$ norm less than or equal to $\|f\|_{\textrm{Lip}-\gamma}$. Then, by lemma \ref{EmbedLip1}, $f$ is $\textrm{Lip}-\gamma'$ on $B(x,\delta)\cap U$ and:
	\[\|f\|_{\textrm{Lip}-\gamma',B(x,\delta)\cap U}\leq\|f\|_{\textrm{Lip}-\gamma}\max\left(1,\sum\limits_{j=\lfloor\gamma'\rfloor+1}^{\lfloor\gamma\rfloor}{\frac{(2\delta)^{j-\gamma'}}{(j-\lfloor\gamma'\rfloor)!}}+(2\delta)^{\gamma-\gamma'}\right)\]
	Using now lemma \ref{LocalLipGen}, we deduce that $f$ is $\textrm{Lip}-\gamma'$ on $U$ with a $\textrm{Lip}-\gamma'$ controlled as follows:
	\[\begin{array}{rl}
	\|f\|_{\textrm{Lip}-\gamma'}\leq&\|f\|_{\textrm{Lip}-\gamma}\max\left(1,\sum\limits_{j=\lfloor\gamma'\rfloor+1}^{\lfloor\gamma\rfloor}{\displaystyle\frac{(2\delta)^{j-\gamma'}}{(j-\lfloor\gamma'\rfloor)!}}+(2\delta)^{\gamma-\gamma'}\right).\\
	&\max\left(1,\max\limits_{0\leq k\leq \lfloor\gamma'\rfloor}{\displaystyle\frac{1}{\delta^{\gamma'-k}}(1+\sum\limits_{j=0}^{\lfloor\gamma'\rfloor-k} \displaystyle\frac{\delta^j}{j!})}\right)\\
	\end{array}\]
	The above inequality holding for every $\delta>0$, we can make it sharper by taking the infinimum of the right-hand side over all possible positive values of $\delta$. This ends the proof.\end{proof}
	\begin{Rem} 
	\[	M_{\gamma,\gamma'}=\inf\limits_{\delta>0} \left\{\max\left(1,\sum\limits_{\lfloor\gamma'\rfloor+1}^{\lfloor\gamma\rfloor}{\frac{(2\delta)^{j-\gamma'}}{(j-\lfloor\gamma'\rfloor)!}}+(2\delta)^{\gamma-\gamma'}\right)
	\max\left(1,\max\limits_{0\leq k\leq \lfloor\gamma'\rfloor}{\frac{1}{\delta^{\gamma'-k}}(1+\sum\limits_{0}^{\lfloor\gamma'\rfloor-k} \frac{\delta^j}{j!})}\right)\right\}
	\]
	By considering the value $\delta=1/2$ in the expression above, we get the following estimate: 
	\[M_{\gamma,\gamma'}\leq 2^{\gamma'}e(1+e^{1/2})\leq 2^{\gamma}e(1+e^{1/2})\]
	which has the additional advantage of being dependent on only one of the variables $\gamma$ or $\gamma'$.
	\end{Rem}
	%%%%%%
	% Smooth functions on open convex sets
	%%%%%%
	%%%%%%
	\subsection{Smooth functions on open convex sets}
	As highlighted for example in \cite{TCL}, a simpler proof of theorem \ref{EmbedLip2} can be given when the domain of definition of the map is open and convex. We recall a characterization of Lipschitz maps in this setting, which also gives a very useful recursive definition of Lipschitzness. The proof of the following is trivial and can be found if needed in \cite{TCL} for example:
	%%%%
	%%%% CARACTERISATION OF LIP OPEN-CONVEX
	%%%%
	\begin{theo}\label{CaracLip} Let $n\in \mathbb{N}$, $0< \varepsilon\leq 1$ and $C\geq 0$. Let $E$ and $F$ be two normed vector spaces and $U$ be a subset of $E$.  We assume that $(E^{\otimes k})_{1\leq k \leq n}$ are endowed with norms satisfying the projective property. Let $f:U\to F$ be a map and for every $k\in [\![1,n]\!]$, let $f^k:U\to \mathcal{L}(E^{\otimes k},F)$ be a map with values in the space of the symmetric $k$-linear mappings from $E$ to $F$. We consider the two following assertions:\begin{description}
	\item[(A1)] $(f,f^1,\ldots,f^n)$ is $\textrm{Lip}-(n+\varepsilon)$ and $\|f\|_{\textrm{Lip}-(n+\varepsilon)}\leq C$.
	\item[(A2)] $f$ is $n$ times differentiable, with $f^1,\ldots,f^n$ being its successive derivatives. $\|f\|_{\infty}$, $\|f^1\|_{\infty}$, $\ldots$, $\|f^n\|_{\infty}$ are upper-bounded by $C$ and for all $x,y\in U: \|f^n(x)-f^n(y)\|\leq C \|x-y\|^{\varepsilon}$.\end{description}
	If $U$ is open then $\mathbf{(A1)}\Rightarrow \mathbf{(A2)}$. If furthermore $U$ is convex then $\mathbf{(A1)}\Leftrightarrow \mathbf{(A2)}$.
	\end{theo}
	The following result about smooth maps is very useful and comes as an easy consequence of theorem \ref{CaracLip}:
	%%%% SMOOTH BOUNDED FUNCTIONS ARE LIP%%%%
	\begin{Cor} Let $n \in \mathbb{N}^*$. A map $f$ defined on a given open convex set that is $n+1$ times continuously differentiable and is such that its derivatives are bounded is Lipschitz-$n$ on that set (assuming that the space of the domain of definition and its successive tensor product spaces are endowed with norms satisfying the projective property). Its Lipschitz-$n$ norm can be upper-bounded by the following constant:
	\[L_n=\max{\{\|f\|_{\infty},\|f^{1}\|_{\infty},\ldots,\|f^{n+1}\|_{\infty}\}}\]
	\end{Cor}
	When the domain of a Lipschitz map is open, convex and bounded, we get a sharper estimate than the one obtained in lemma \ref{EmbedLip1}:
	%%%% EMBEDDING OF LIP OPEN-CONVEX BOUNDED SET%%%%
	\begin{lemma}\label{EmbedLip11} Let $\gamma,\gamma'>0$ such that $\gamma'\leq \gamma$. Let $E$ and $F$ be two normed vector spaces and $U$ be an open convex bounded subset of $E$. Let $f:U\to F$ be a $\textrm{Lip}-\gamma$ function. We assume that $(E^{\otimes k})_{1\leq k \leq \lfloor\gamma\rfloor}$ are endowed with norms satisfying the projective property. Then $f$ is $\textrm{Lip}-\gamma'$ and if $L\geq0$ is larger than or equal to the diameter of $U$ then:
\[\|f\|_{\textrm{Lip}-\gamma'}\leq\|f\|_{\textrm{Lip}-\gamma}\max\left(1,L^{\min(\lfloor\gamma'\rfloor+1,\gamma)-\gamma'}\right)\]
	\end{lemma}
	\begin{proof} Uses the characterization in lemma \ref{CaracLip} and, if $\lfloor\gamma'\rfloor<\lfloor\gamma\rfloor$, the fundamental theorem of calculus.\end{proof}
	Always in the case of an open convex domain, we also get a sharper control of the Lipschitz norm from the uniform local behaviour of the map:
	%%%% LOCAL CARACTERISATION OF LIP OPEN-CONVEX%%%%
	\begin{lemma}\label{LocalLip} Let $\gamma>0$. Let $E$ and $F$ be two normed vector spaces and $U$ be an open convex subset of $E$. We assume that $(E^{\otimes k})_{1\leq k \leq \lfloor\gamma\rfloor}$ are endowed with norms satisfying the projective property. Let $f:U\to F$ be a map such that there exists $\delta>0$ and $C\geq0$ such that, for every $x\in U$, $f_{|B(x,\delta)\cap U}$ is $\textrm{Lip}-\gamma$ with a norm less than or equal to $C$. Then $f$ is $\textrm{Lip}-\gamma$ and:
	\[\|f\|_{\textrm{Lip}-\gamma}\leq C\max\left(1,\frac{2}{\delta^{\gamma-\lfloor\gamma\rfloor}}\right)\] 
	\end{lemma}
	\begin{proof} Uses the characterization in theorem \ref{CaracLip} and the same technique as in the proof of lemma \ref{LocalLipGen}. If necessary, a complete proof can be found for example in \cite{TCL}.\end{proof}
	Theorem \ref{EmbedLip2} now becomes:
	%%%% GENERAL THRM: EMBEDDINGS OF LIP OPEN-CONVEX %%%%
	\begin{theo}\label{EmbedLip22} Let $\gamma,\gamma'>0$ such that $\gamma'<\gamma$. Let $E$ and $F$ be two normed vector spaces and $U$ be an open convex subset of $E$. We assume that $(E^{\otimes k})_{1\leq k \leq \lfloor\gamma\rfloor}$ are endowed with norms satisfying the projective property. Let $f:U\to F$ be a $\textrm{Lip}-\gamma$ map. Then $f$ is $\textrm{Lip}-\gamma'$ and:
	\[\|f\|_{\textrm{Lip}-\gamma'}\leq 4 \|f\|_{\textrm{Lip}-\gamma}\]
	\end{theo}
	\begin{proof}
	The proof comes off as an easy corollary of lemmas \ref{EmbedLip1} and \ref{LocalLip}. Using the same technique as in the proof of theorem \ref{EmbedLip2}, we show that:
	\[\|f\|_{\textrm{Lip}-\gamma'}\leq m_{\gamma,\gamma'} \|f\|_{\textrm{Lip}-\gamma}\]
	where:
	\[m_{\gamma,\gamma'}=\inf_{\delta>0} \max\left(1,(2\delta)^{\min(\lfloor\gamma'\rfloor+1,\gamma)-\gamma'}\right)\max\left(1,\frac{2}{\delta^{\gamma'-\lfloor\gamma'\rfloor}}\right)\]
	and one easily gets $ m_{\gamma,\gamma'}\leq 4$.
	\end{proof}	
	%%%%%%
	% Composition of Lipschitz functions
	%%%%%%
	%%%%%%
	\subsection{Composition of Lipschitz functions}
	%%%%%%
	% Composition with linear maps
	%%%%%%
	\subsubsection{Composition with linear maps}
	As one would expect, a well-defined composition of two Lipschitz maps is also Lipschitz. We start first with the simple case where one of the maps is linear as the derivatives are easier to extract, though, technically, a continuous linear map defined on the whole space is not necessarily Lipschitz (as its values are not uniformly bounded unless it is null).
	%%%%
	%%%% COMPOSITION WITH LINEAR FUNCTIONS 
	%%%%
	\begin{Prop}\label{Compolinfunc} Let $E$, $F$ and $G$ be three normed vector spaces and $U$ be a subset of $E$. Let $\gamma>0$ and let $f:U \to F$ be a $\textrm{Lip}-\gamma$ map. Let $u: F\to G$ a bounded linear map. Then $u\circ f$ is $\textrm{Lip}-\gamma$ and $\|u\circ f\|_{\textrm{Lip}-\gamma} \leq \|u\|\|f\|_{\textrm{Lip}-\gamma}$.
	\end{Prop}
	\begin{proof} Let $n\in \mathbb{N}$ such that $\gamma \in (n,n+1]$. Let $f^1,\ldots,f^n$ be maps on $U$ such that $(f,f^1,\ldots,f^n)$ is $\textrm{Lip}-\gamma$ and let $R_0,\ldots,R_n$ be the associated remainders. Let $g=u\circ f$ and for every $k\in[\![0,n]\!]$, let $g^k$ and $S_k$ be defined as follows:
	\[\forall x,y \in E, \forall v\in E^{\otimes k}: g^k(x)(v)=u(f^k(x)(v)),\quad S_k(x,y)(v)=u(R_k(x,y)(v))\]
	Then it is easy to check that $(g,g^1,\ldots,g^n)$ is $\textrm{Lip}-\gamma$ with $S_0,\ldots,S_n$ as remainders and with a $\textrm{Lip}-\gamma$ norm upper-bounded by $\|u\|\|f\|_{\textrm{Lip}-\gamma}$.\end{proof}
	\begin{Rem} Although a linear map in general is not Lipschitz, we can restrict ourselves, in the previous proposition, to a bounded domain of $F$ so that the restriction of $u$ on that domain is Lipschitz. We will be then in the case of a composition of two Lipschitz maps but we don't get a control of the Lipschitz norm as sharp as the one in proposition \ref{Compolinfunc}.
	\end{Rem}
	\begin{Rem} The second item of proposition \ref{ProductLipFunc} now becomes a special case of proposition \ref{Compolinfunc}. The condition on the norms on the Cartesian product of vector spaces is equivalent to the continuity of the (linear) projection map onto one of these spaces.
	\end{Rem}
	%%%%
	%%%% COMPOSITION WITH LINEAR FUNCTIONS II
	%%%%
	\begin{Prop}\label{Compolinfunc2} Let $\gamma>0$ and $E$, $F$ and $G$ be three normed vector spaces. We assume that $(E^{\otimes k})_{1 \leq k\leq \lfloor\gamma\rfloor}$ and $(F^{\otimes k})_{1 \leq k\leq \lfloor\gamma\rfloor}$ are endowed with compatible norms. Let $f:F \to G$ be a $\textrm{Lip}-\gamma$ map and $u: E\to F$ a bounded linear map. Then $f\circ u$ is $\textrm{Lip}-\gamma$ and $\|f\circ u\|_{\textrm{Lip}-\gamma} \leq \|f\|_{\textrm{Lip}-\gamma} \max(1,\|u\|^{\gamma})$.
	\end{Prop}
	\begin{proof} Let $f^1,\ldots,f^n$ be maps defined on $F$ such that $(f,f^1,\ldots,f^n)$ is $\textrm{Lip}-\gamma$ and let $R_0,\ldots,R_n$ be the associated remainders.
	Let $g=f\circ u$ and, for $0\in[\![1,n]\!]$, let $g^k:E\to \mathcal{L}(E^{\otimes k},G)$ and $S_k:E\times E\to\mathcal{L}(E^{\otimes k},G)$ be the maps defined by:
	\[\forall x,y\in E, \forall v \in E^{\otimes k}: g^k(x)(v)=f^k(u(x))(u^{\otimes k}(v)), S_k(x,y)(v)=R_k(u(x),u(y))(u^{\otimes k}(v))\]
	Let $k\in[\![0,n]\!]$, $x,y\in E$ and $v \in E^{\otimes k}$. Then we have, using the previous definitions and the Taylor expansion of $f$:
	\[\begin{array}{rcl}
	g^k(x)(v)&=&f^k(u(x))(u^{\otimes k}(v))\\
			&=&\sum\limits_{j=k}^{n} f^j(u(y))\left(\frac{\displaystyle u^{\otimes k}(v)\otimes (u(x-y))^{\otimes j-k}}{\displaystyle (j-k)!}\right)+R_k(u(x),u(y))(u^{\otimes k}(v))\\
			&=&\sum\limits_{j=k}^{n} f^j(u(y))\left(u^{\otimes j}(\frac{\displaystyle v\otimes (x-y)^{\otimes(j-k)}}{\displaystyle(j-k)!})\right)+S_k(x,y)(v)\\
			&=&\sum\limits_{j=k}^{n} g^j(y)\left(\frac{\displaystyle v\otimes (x-y)^{\otimes(j-k)}}{\displaystyle(j-k)!}\right)+S_k(x,y)(v)\\
	\end{array}\]
	Moreover $\|g^k(x)\|\leq \|f\|_{\textrm{Lip}-\gamma}\|u\|^k$ and $\|S_k(x,y)\|\leq \|f\|_{\textrm{Lip}-\gamma}\|u\|^\gamma \|x-y\|^{\gamma -k}$. Hence, $(g,g^1,\ldots,g^n)$ is $\textrm{Lip}-\gamma$ (with $(S_0,\ldots, S_n)$ as remainders) and
	\[\begin{array}{rcl}
	\|g\|_{\textrm{Lip}-\gamma}&\leq& \max(\|f\|_{\textrm{Lip}-\gamma},\|f\|_{\textrm{Lip}-\gamma}\|u\|,\ldots,\|f\|_{\textrm{Lip}-\gamma}\|u\|^n,\|f\|_{\textrm{Lip}-\gamma}\|u\|^\gamma)\\
					&\leq&\|f\|_{\textrm{Lip}-\gamma} \max(1,\|u\|^{\gamma})\\
	\end{array}\]\end{proof}
	%%%%
	%%%% Composition with linear fcts II with no compatible norms
	%%%%
	\begin{Rem} If $(E^{\otimes k})_{1 \leq k\leq \lfloor\gamma\rfloor}$ and $(F^{\otimes k})_{1 \leq k\leq \lfloor\gamma\rfloor}$ are not necessarily endowed with compatible norms, then $f\circ u$ is still $\textrm{Lip}-\gamma$ and:
	\[\|f\circ u\|_{\textrm{Lip}-\gamma} \leq \|f\|_{\textrm{Lip}-\gamma} \max_{0\leq k \leq \lfloor\gamma\rfloor} \|u^{\otimes k}\| (1\vee \|u\|^{\gamma-k})\]
	\end{Rem}
	%%%%%%
	% Formal derivatives
	%%%%%%
	\subsubsection{Formal derivatives}	
	Before proceeding to the proof that the composition of two Lipschitz maps is indeed Lipschitz, we will need a few combinatorial results along with the identification of higher derivatives of the composition of two maps and a general recursive criterion for a map to be Lipschitz. The contents of this subsection will only be useful to us to obtain the results of the next one, i.e. to show that the composition of two Lipschitz maps is indeed Lipschitz.\\
	Let $E$ and $F$ be two vector spaces. let $f:U \to F$ be a map defined on a subset $U$ of $E$ and for every $k \in [\![1,n]\!]$, $n$ being a positive integer, let  $f^k:U\to \mathcal{L}_s(E^{\otimes k},F)$ be a map. Here, the $f^k$'s play \emph{formally} the role of the $k$\textsuperscript{th} derivative of $f^0$. As, for $k\in [\![1,n]\!]$, $f^k(x)$ is a symmetric map for every $x\in U$, we can identify $f^k$ and $(f^1)^{k-1}$ in the following natural way. For every $x\in U$ and $v_1,\ldots, v_{k+1} \in E$:
	\[f^k(x)(v_1\otimes\cdots\otimes v_{k+1})=(f^1)^{k-1}(x)(v_1\otimes\cdots\otimes v_{k})(v_{k+1})\]
	From this identification, we can easily see $(f^1, (f^1)^1, \ldots, (f^1)^{n-1})$ as a $\textrm{Lip}-(n+\varepsilon-1)$ map with norm upper-bounded by that of $f$:
	%%%%
	%%%% Control f with f'
	%%%%
	\begin{lemma}\label{RecursiveLipCtrl} Let $n\in \mathbb{N}^*$ and $0< \varepsilon\leq 1$. Let $E$ and $F$ be two normed vector spaces and $U$ be a subset of $E$. let $f:U \to F$ be a map and for every $k \in [\![1,n]\!]$, let $f^k:U\to \mathcal{L}_s(E^{\otimes k},F)$ be a map with values in the space of the symmetric $k$-linear mappings from $E$ to $F$. Denote by $R_0:E\times E\to F$ the remainder of order $0$ associated to the collection $f=(f^0,f^1,\ldots,f^n)$. 
	\begin{itemize}
	\item If $f$ is Lipschitz-$(n+\varepsilon)$ on $U$ and there exists $a>0$ such that:
	\[\forall k\in [\![0,n-1]\!], \forall v\in E^{\otimes k}, \forall \tilde{v}\in E: \quad \|v \otimes \tilde{v}\|\leq a \|v\|\|\tilde{v}\|\]
	then $(f^1, (f^1)^1, \ldots, (f^1)^{n-1})$ is $\textrm{Lip}-(n+\varepsilon-1)$ and:
	\[\|f^1\|_{\textrm{Lip}-(n+\varepsilon-1)} \leq a \|f\|_{\textrm{Lip}-(n+\varepsilon)}\]
	\item If:
		\begin{enumerate}
		\item $(f^1, (f^1)^1, \ldots, (f^1)^{n-1})$ is $\textrm{Lip}-(n+\varepsilon-1)$.
		\item There exists a constant $M$ such that:
	\[\forall x,y \in U:\quad \|R_0(x,y)\|\leq M\|x-y\|^{n+\varepsilon}\]
	(Denote by $\|R_0\|_{\infty}$ the smallest value for such a constant $M$.)
		\item $f^0$ is bounded.
		\item There exists $b>0$ such that:
	\[\forall k\in [\![0,n-1]\!], \forall v\in E^{\otimes k}, \forall \tilde{v}\in E: \quad  \|v\|\|\tilde{v}\| \leq b \|v \otimes \tilde{v}\|\]
		\end{enumerate}
	then $f$ is Lipschitz-$(n+\varepsilon)$ on $U$ and:
		\[\|f\|_{\textrm{Lip}-(n+\varepsilon)}\leq \max{(b\|f^1\|_{\textrm{Lip}-(n+\varepsilon-1)},\|R_0\|_{\infty},\|f^0\|_{\infty})}\]
	\end{itemize}
	\end{lemma}	
	\begin{Rem}\label{FiniteDimIsometryTensorNorm} In the finite-dimensional case, the conditions on the norms stated in lemma \ref{RecursiveLipCtrl} are not an issue and can even be obtained with constants $a=b=1$ for some of the examples of norms provided in the subsection \ref{SectionDefinitionNorms}.
	\end{Rem}
	%%%
	%%% Formal Derivative of bilinear
	%%%
	\begin{Def} Let $E$, $F$, $G$ and $H$ be normed vector spaces and $U$ be a subset of $E$. Let $n\in \mathbb{N}^*$. Let $f:U \to F$ and $g:U \to G$ be two maps and for every $k \in [\![1,n]\!]$, let $f^k:U\to \mathcal{L}_s(E^{\otimes k},F)$  and $g^k:U\to \mathcal{L}_s(E^{\otimes k},G)$ be any two maps. Let $B: F\times G\to H$ a bilinear map. For $k\in[\![1,n]\!]$,  we call the $k$\textsuperscript{th} bilinear derivative of the map $B(f,g)$ (where $f$ and $g$ are identified with the collections $(f,f^1,\ldots,f^n)$ and $(g,g^1,\ldots,g^n)$ respectively), the map defined on $U$ with values in $\mathcal{L}_s(E^{\otimes k},H)$ obtained by formally differentiating $k$ times the map $B(f,g)$ , i.e. for $x\in U$ and  $v\in E^{\otimes k}$:
	\[ B(f,g)^k(x)(v)=\underset{\sigma\in\mathcal{S}_k}{\underset{i\in[\![0,k]\!]}\sum}{\frac{B(f^i(x),g^{k-i}(x))}{i!(k-i)!}\sigma(v)}\]
	\end{Def}
	
	In the above definition, for $i,j\in[\![0,n]\!]$, $x,y \in U$, $B(f^i(x),g^{j}(y))$ is understood to be the unique linear map defined on $E^{\otimes(i+j)}$ by the following:
	\[\forall v_1,\ldots,v_{i+j} \in E: B(f^i(x),g^{j}(y))(v_1\otimes\cdots\otimes v_{i+j})=B(f^i(x)(v_1\otimes\cdots\otimes v_{i}),g^{j}(y)(v_{i+1}\otimes\cdots\otimes v_{i+j}))\]
	We check first that this definition is stable under successive derivations and compatible with the identification between the derivatives of a map and those of its first derivative. This is essential if we are to use an induction argument to show that the image of Lipschitz maps by a bilinear map is also Lipschitz.
	
	\begin{Prop}\label{IDFormalDeriBilinear} Let $n\in \mathbb{N}^*$. Let $E$, $F$, $G$ and $H$ be normed vector spaces and $U$ be a subset of $E$. Let $f:U \to F$ and $g:U \to G$ be two maps and for every $k \in [\![1,n]\!]$, let $f^k:U\to \mathcal{L}_s(E^{\otimes k},F)$  and $g^k:U\to \mathcal{L}_s(E^{\otimes k},G)$ be any two maps. Let $B: F\times G\to H$ a bilinear map. Then, for $k \in [\![1,n]\!]$, the $k$\textsuperscript{th} bilinear derivative of the map $B(f,g)$ can be identified with the $(k-1)$\textsuperscript{th} bilinear derivative of the $1$\textsuperscript{st} bilinear derivative of the map $B(f,g)$.
	\end{Prop}
	\begin{proof}
	For shorter formulas, let us denote $Z=B(f,g)$. Let $k\in [\![0,n-1]\!]$ and let us make the identification between $(Z^1)^k$ and $Z^{k+1}$. $Z^1$ can be written as a sum of the image by bilinear maps of maps defined on $U$ and therefore $(Z^1)^k$ can be defined using formal bilinear derivation. More precisely, for $x\in U$ and $v\in E^{\otimes k}$:
	\[(Z^1)^k(x)(v)=\underset{\sigma\in\mathcal{S}_k}{\underset{i\in[\![0,k]\!]}\sum}{\frac{B((f^1)^i(x),g^{k-i}(x))+B(f^i(x),(g^1)^{k-i}(x))}{i!(k-i)!}\sigma(v)}\]
	Let $v_1,\ldots,v_k,v_{k+1}\in E$ and define: $v=v_1\otimes\cdots\otimes v_k$. Studying the position of $v_{k+1}$ in $Z^{k+1}(x)(v\otimes v_{k+1})$, we are naturally led into dividing the sum into the two following parts:
	\[\begin{array}{rcl}
	Z^{k+1}(x)(v\otimes v_{k+1})&=&\underset{\sigma\in\mathcal{S}_{k+1}}{\underset{i\in[\![0,k+1]\!]}\sum}{\frac{B(f^i(x)(\sigma^{1,i}(v\otimes v_{k+1})),g^{k+1-i}(x)(\sigma^{2,i}(v\otimes v_{k+1})))}{i!(k+1-i)!}}\\
	&=&\sum\limits_{i=1}^{k+1}{\left(\underset{\sigma^{-1}(k+1)\leq i}{\underset{\sigma\in\mathcal{S}_{k+1}}\sum}{\frac{B(f^i(x)(\sigma^{1,i}(v\otimes v_{k+1})),g^{k+1-i}(x)(\sigma^{2,i}(v\otimes v_{k+1})))}{i!(k+1-i)!}}\right)}+\\
	&&\sum\limits_{i=0}^{k}{\left(\underset{\sigma^{-1}(k+1)>i}{\underset{\sigma\in\mathcal{S}_{k+1}}\sum}{\frac{B(f^i(x)(\sigma^{1,i}(v\otimes v_{k+1})),g^{k+1-i}(x)(\sigma^{2,i}(v\otimes v_{k+1})))}{i!(k+1-i)!}}\right)}\\
	\end{array}\]
	For every $\sigma \in\mathcal{S}_{k+1}$ let $\tau_\sigma\in\mathcal{S}_{k}$ be defined as follows:
	\[\tau_{\sigma}(j)
	=\left\{\begin{array}{ll}
	\sigma(j)&,\textrm{ if } j<\sigma^{-1}(k+1)\\
	\sigma(j+1)&,\textrm{ if } j\geq\sigma^{-1}(k+1)\\	
	\end{array}\right.
	\]
	The map $\sigma \mapsto \tau_\sigma$ is surjective and for each $\tau\in\mathcal{S}_{k}$, there exists exactly $(k+1)$ elements $\sigma \in\mathcal{S}_{k+1}$ such that $\tau=\tau_\sigma$. More precisely, for $i\in[\![0,k+1]\!]$ and $\tau\in\mathcal{S}_{k}$:
	\[\textrm{card}\{\sigma \in\mathcal{S}_{k+1}: \tau=\tau_\sigma, \sigma^{-1}(k+1)\leq i\}=i\]	
	and
	\[\textrm{card}\{\sigma \in\mathcal{S}_{k+1}: \tau=\tau_\sigma, \sigma^{-1}(k+1)> i\}=k+1-i\]
	Let $i\in[\![1,k+1]\!]$. Since $f^i(x)$ is symmetric then, for every  $\sigma\in\mathcal{S}_{k+1}$ such that $\sigma^{-1}(k+1)\leq i$, we have:
	\[	f^i(x)(\sigma^{1,i}(v\otimes v_{k+1}))=f^i(x)(\tau_\sigma^{1,i-1}(v)\otimes v_{k+1})\]
	which gives:
	\[\begin{array}{l}
	\underset{\sigma^{-1}(k+1)\leq i}{\underset{\sigma\in\mathcal{S}_{k+1}}\sum}{\displaystyle\frac{B(f^i(x)(\sigma^{1,i}(v\otimes v_{k+1})),g^{k+1-i}(x)(\sigma^{2,i}(v\otimes v_{k+1})))}{i!(k+1-i)!}}=\\
	{\underset{\tau\in\mathcal{S}_{k}}\sum}{\displaystyle\frac{B((f^1)^{i-1}(x)(\tau^{1,i-1}(v))(v_{k+1}),g^{k+1-i}(x)(\tau^{2,i-1}(v)))}{(i-1)!(k+1-i)!}}
		\end{array}\]
	Summing over all $i\in[\![1,k+1]\!]$:
	\[\begin{array}{rl}
	&\sum\limits_{i=1}^{k+1}\underset{\sigma^{-1}(k+1)\leq i}{\underset{\sigma\in\mathcal{S}_{k+1}}\sum}{\displaystyle\frac{B(f^i(x)(\sigma^{1,i}(v\otimes v_{k+1})),g^{k+1-i}(x)(\sigma^{2,i}(v\otimes v_{k+1})))}{i!(k+1-i)!}}\\
	=&\sum\limits_{j=0}^{k}{\underset{\tau\in\mathcal{S}_{k}}\sum}{\displaystyle\frac{B((f^1)^{j}(x)(.)(v_{k+1}),g^{k-j}(x))}{j!(k-j)!}}(\tau(v))
		\end{array}\]	
	We deal with the other term by using a similar idea. We finally get that $Z^{k+1}(x)(v\otimes v_{k+1})$ is equal to:
	\[\sum\limits_{j=0}^{k}{\underset{\tau\in\mathcal{S}_{k}}\sum}{\displaystyle\frac{B((f^1)^{j}(x)(.)(v_{k+1}),g^{k-j}(x))+B(f^{j}(x),(g^1)^{k-j}(x)(.)(v_{k+1}))}{j!(k-j)!}}(\tau(v))\]
	Which is exactly $(Z^1)^k(x)(v)(v_{k+1})$.
	\end{proof}
	%%%
	%%% Chain rule FORMAL derivative
	%%%
	\begin{Def} Let $n\in \mathbb{N}^*$. Let $E$, $F$ and $G$ be three normed vector spaces. Let $U$ be a subset of $E$ and $V$ be a subset of $F$. Let $f:U \to F$ and $g:V \to G$ be two maps such that $f(U) \subseteq V$; and for every $k \in [\![1,n]\!]$, let $f^k:U\to \mathcal{L}_s(E^{\otimes k},F)$  and $g^k:V\to \mathcal{L}_s(F^{\otimes k},G)$ be any two maps. For $k\in [\![1,n]\!]$, we call the $k$\textsuperscript{th} chain rule derivative of the composition $g\circ f$ (where $f$ and $g$ are identified with the collections $(f,f^1,\ldots,f^n)$ and $(g,g^1,\ldots,g^n)$ respectively), the map defined on $U$ with values in $\mathcal{L}_s(E^{\otimes k},G)$ obtained by formally applying the chain rule on $g\circ f$, i.e. for every $y\in U$, and $v\in E^{\otimes k}$, $(g\circ f)^k(y)(v)$ is given by the following formula:
	\[(g\circ f)^k(x)(v)=\sum\limits_{j=1}^k \frac{g^{j}(f(x))}{j!} \underset{\underset{i_1+\cdots+i_j=k}{1\leq i_1,\ldots,i_j\leq n}}\sum \frac{f^{i_1}(x)\otimes \cdots\otimes f^{i_j}(x)}{i_1!\cdots i_j!}\left(\sum\limits_{\sigma\in\mathcal{S}_k}\sigma(v)\right)\]
	\end{Def}
	%%%
	%%% Linear examples of Chain rule FORMAL derivatives
	%%%
	\begin{example} Given a Lipschitz map $f$ and by naturally identifying a linear map $u$ with the collection $(u,x\mapsto u, 0,\ldots,0)$, we note that the successive chain rule derivatives correspond to the suggested representation of $u\circ f$ and $f\circ u$ as Lipschitz maps in the proofs of propositions \ref{Compolinfunc} and \ref{Compolinfunc2} respectively.
	\end{example}
	%%%
	%%% Symmetric identity symbol
	%%%
	\begin{Notation} We introduce the symbol $\overset{S}=$ to say that two expressions have the same symmetric part (in the appropriate framework).
	\end{Notation}
	We will need a couple of combinatorial results before we can proceed:
	%%%
	%%% Symmetric identity 1
	%%%
	\begin{lemma}\label{SigmaReasoning1}  Let $n\in \mathbb{N}^*$. Let $E$ and $F$ be two normed vector spaces and for every $k \in [\![1,n]\!]$, let $f^k \in\mathcal{L}_s(E^{\otimes k},F)$. Let $k\in [\![1,n-1]\!]$ and $p\in[\![2,k+1]\!]$, then, for all $v_1,\ldots,v_{k+1} \in E$:
	\[A(k,p):=\frac{1}{p}\underset{\underset{m_1+\cdots+m_p=k+1}{1\leq m_1,\ldots,m_p\leq n}}\sum 
	\frac{f^{m_1}\otimes \cdots\otimes f^{m_p}}{m_1!\cdots m_p!}\sum\limits_{\sigma\in\mathcal{S}_{k+1}}v_{\sigma(1)}\otimes\cdots\otimes v_{\sigma(k+1)}\]
	and
	\[
	B(k,p):=\sum\limits_{i=p-1}^k
	\underset{\underset{m_1+\cdots+m_{p-1}=i}{1\leq m_1,\ldots,m_{p-1}\leq n}}\sum 
	\frac{f^{m_1}\otimes \cdots\otimes f^{m_{p-1}}\otimes f^{k-i+1}}{m_1!\cdots m_{p-1}!(k-i)!}
	{\sum\limits_{\sigma\in\mathcal{S}_k}}v_{\sigma(1)}\otimes\cdots\otimes  v_{\sigma(k)}\otimes v_{k+1}\]
	have the same symmetric parts (i.e. $A(k,p)\overset{S}=B(k,p)$).
	\end{lemma}
	\begin{proof} Define:
	\[I=\{(m_1,\ldots,m_p,\sigma)|\quad 1\leq m_1,\ldots,m_p\leq n, m_1+\cdots+m_p=k+1, \sigma \in \mathcal{S}_{k+1}\}\]
	For $i\in [\![1,p]\!]$ and $m\in [\![1,k-p+2]\!]$, define $J_{i,m}$ as being the set:
	\[
	J_{i,m}:=\left\{(m_1,\ldots,m_p,\sigma) \in I |\quad  m_i=m, \sum_{j=1}^{i-1}m_j +1 \leq \sigma^{-1}(k+1) \leq \sum_{j=1}^{i}m_j \right\}
	\]
	and for $r\in[\![1,m]\!]$, we define:
	\[J_{i,m}^r=
	\left\{(m_1,\ldots,m_p,\sigma) \in J_{i,m} |\quad  \sigma^{-1}(k+1) = \sum_{j=1}^{i-1}m_j+r \right\}
	\]
	It is clear that the sets  $(J_{i,m})_{1\leq i \leq p,1\leq m\leq k-p+ 2}$ form a partition of $I$. \\
	Let $i\in [\![1,p]\!]$ and $m\in [\![1,k-p+2]\!]$. Let $(m_1,\ldots,m_p,\sigma) \in J_{i,m}$. Define $\eta_{\sigma} \in \mathcal{S}_k$ as follows: 
	\[\left\{\begin{array}{ll}
	\eta_{\sigma}(r)=\sigma(r)&, \forall r\in[\![1,\sum_{1}^{i-1}m_j]\!]\\
	\eta_{\sigma}(r)=\sigma(r+m)&, \forall r\in[\![\sum_{1}^{i-1}m_j+1,k+1-m]\!]\\
	\eta_{\sigma}(r)=\sigma(r-\sum_{i+1}^{p}m_j)&, \forall r\in[\![k+2-m,k-\sum_{1}^{i}m_j+\sigma^{-1}(k+1)]\!]\\
	\eta_{\sigma}(r)=\sigma(r-\sum_{i+1}^{p}m_j+1)&, \forall r\in[\![k-\sum_{1}^{i}m_j+\sigma^{-1}(k+1)+1,k]\!]\\
	\end{array}\right.
	\]	
	i.e.
	\[
	\begin{array}{ll}
	\eta_{\sigma}=&(\sigma(1), \ldots, \sigma(\sum_{1}^{i-1}m_j),
										\sigma(\sum_{1}^{i}m_j+1),\ldots,\sigma(k+1),\\
										&\sigma(\sum_{1}^{i-1}m_j+1), \ldots, \sigma(\sigma^{-1}(k+1)-1),
										\sigma(\sigma^{-1}(k+1)+1), \ldots, \sigma(\sum_{1}^{i}m_j))\\
	\end{array}\]
	then, as $ f^{m}$ is symmetric, we have:
	\[\begin{array}{rl}
	&f^{m_1}\otimes \cdots\otimes f^{m_p} (v_{\sigma(1)}\otimes\cdots\otimes v_{\sigma(k+1)})\\
	\overset{S}=&
	f^{m_1}\otimes \cdots\otimes f^{m_{i-1}}\otimes f^{m_{i+1}}\otimes\cdots\otimes f^{m_p}\otimes f^{m}
	(v_{\eta_{\sigma}(1)}\otimes\cdots\otimes v_{\eta_{\sigma}(k)} \otimes v_{k+1})\\
	\end{array}\]
	Let $r\in[\![1,m]\!]$. As the map:
	\[\begin{array}{rlcl}
	\varphi_{i,m}^r:&J_{i,m}^{r}&\rightarrow&\{(m_1,\ldots,m_{p-1})\in[\![1,n]\!]^{(p-1)}, \sum_1^{p-1} m_l=k+1-m\}\times \mathcal{S}_{k}\\
	&(m_1,\ldots,m_p,\sigma)&\mapsto&((m_1,\ldots,m_{i-1},m_{i+1},\ldots,m_{p}),\eta_{\sigma})\\
	\end{array}
	\]
	is bijective, we have, by virtue of the identity above:
	\[\begin{array}{rl}
	& \underset{(m_1,\ldots,m_p,\sigma) \in J_{i,m}^r}\sum 
	f^{m_1}\otimes \cdots\otimes f^{m_p}( v_{\sigma(1)}\otimes\cdots\otimes v_{\sigma(k+1)})\\
	\overset{S}=&
	\underset{\underset{m_1+\cdots+m_{p-1}=k+1-m}{1\leq m_1,\ldots,m_{p-1}\leq n}}\sum 	f^{m_1} \otimes\cdots\otimes f^{m_{p-1}}\otimes f^{m}
	({\sum\limits_{\sigma\in\mathcal{S}_k}} v_{\sigma(1)}\otimes\cdots\otimes  v_{\sigma(k)}\otimes v_{k+1})\\
	\end{array}\]	
	The sets $(J_{i,m}^r)_{1\leq r \leq m}$ form a partition of $J_{i,m}$. Therefore:
	\[\begin{array}{rl}
	& \underset{(m_1,\ldots,m_p,\sigma) \in J_{i,m}}\sum 
	f^{m_1}\otimes \cdots\otimes f^{m_p} (v_{\sigma(1)}\otimes\cdots\otimes v_{\sigma(k+1)})\\
	\overset{S}=&m 
	\underset{\underset{m_1+\cdots+m_{p-1}=k+1-m}{1\leq m_1,\ldots,m_{p-1}\leq n}}\sum 	f^{m_1} \otimes\cdots\otimes f^{m_{p-1}}\otimes f^{m}
	(\sum\limits_{\sigma\in\mathcal{S}_k} v_{\sigma(1)}\otimes\cdots\otimes  v_{\sigma(k)}\otimes v_{k+1})\\
	\end{array}\]
	which finally gives:
	\[\begin{array}{rl}
	A(k,p)\overset{S}=&
	\frac{1}{p}\sum\limits_{i=1}^p\sum\limits_{m=1}^{k-p+2} m 
	\underset{\underset{m_1+\cdots+m_{p-1}=k+1-m}{1\leq m_1,\ldots,m_{p-1}\leq n}}\sum 
	\frac{f^{m_1} \otimes\cdots\otimes f^{m_{p-1}}\otimes f^{m}}{m_1!\cdots m_{p-1}!m!}
	(\sum\limits_{\sigma\in\mathcal{S}_k} v_{\sigma(1)}\otimes\cdots\otimes  v_{\sigma(k)}\otimes v_{k+1})\\
	=&
	\sum\limits_{m=1}^{k-p+2} 
	\underset{\underset{m_1+\cdots+m_{p-1}=k+1-m}{1\leq m_1,\ldots,m_{p-1}\leq n}}\sum 
	\frac{f^{m_1} \otimes\cdots\otimes f^{m_{p-1}}\otimes f^{m}}{m_1!\cdots m_{p-1}!(m-1)!}
	(\sum\limits_{\sigma\in\mathcal{S}_k} v_{\sigma(1)}\otimes\cdots\otimes  v_{\sigma(k)}\otimes v_{k+1})\\
	=&B(k,p)\\
	\end{array}\]
		\end{proof}
		
	\begin{Notation} For any finite set $\{\alpha_1,\ldots,\alpha_r\}$, $\mathcal{S}_{\alpha_1,\ldots,\alpha_r}$ denotes the set of all bijections from $\{\alpha_1,\ldots,\alpha_r\}$ onto itself.
	\end{Notation}	
	
	\begin{lemma}\label{SigmaReasoning2}
	Let $k\in\mathbb{N}^*$, $1\leq i \leq k$ and $v_1,\ldots v_k$ be any letters. Then:
	\[i!{\sum\limits_{\sigma\in\mathcal{S}_k}}v_{\sigma(1)}\otimes\cdots\otimes  v_{\sigma(k)}=
	{\underset{\tau\in\mathcal{S}_{\sigma(1),\ldots,\sigma(i)}}{\sum\limits_{\sigma\in\mathcal{S}_k}}}v_{\tau(\sigma(1))}\otimes\cdots\otimes v_{\tau(\sigma(i))}\otimes v_{\sigma(i+1)}\otimes\cdots\otimes v_{\sigma(k)}
	\]
	\end{lemma}
	\begin{proof} Define the set:
	\[\mathcal{S}^{(i)} =\{(\sigma, \tau)| \sigma \in \mathcal{S}_{k}, \tau\in\mathcal{S}_{\sigma(1),\ldots,\sigma(i)}\}\]
	and the map: 
	\[\begin{array}{rccl}
	\mathcal{J}:& \mathcal{S}^{(i)} & \longrightarrow & \mathcal{S}_k \\
				& (\sigma, \tau) & \longrightarrow & (\tau(\sigma(1)),\ldots,\tau(\sigma(i)),\sigma(i+1),\ldots, \sigma(k))\\
	\end{array}
	\]
	Then $\mathcal{J}$ is well-defined, surjective, and for every $ \sigma \in \mathcal{S}_{k}$, $\mathrm{card}\,\mathcal{J}^{-1}(\{\sigma\})= i!$. The rest of the proof follows immediately.
	\end{proof}
	
	We check now that the chain rule derivation is homogeneous with the bilinear derivation:
	
	%%%% Higher derivatives of composition agree
	\begin{lemma}\label{HighDerivAgree} Let $n\in \mathbb{N}^*$. Let $E$, $F$ and $G$ be three normed vector spaces. Let $U$ be a subset of $E$ and $V$ be a subset of $F$. Let $f:U \to F$ and $g:V \to G$ be two maps such that $f(U) \subseteq V$; and for every $k \in [\![1,n]\!]$, let $f^k:U\to \mathcal{L}_s(E^{\otimes k},F)$  and $g^k:V\to \mathcal{L}_s(F^{\otimes k},G)$ be any two maps. Let $k\in [\![1,n]\!]$. Then the $k$\textsuperscript{th} chain rule derivative of $g\circ f$ and the $(k-1)$\textsuperscript{th} bilinear derivative of $(g\circ f)^1$ agree.
	\end{lemma}
	\begin{proof}
	We first write $(g\circ f)^1$ as the bilinear image of two maps: $(g\circ f)^1=\psi(g^1 \circ f^0,f^1)$; where:
	\[\begin{array}{rccc}
	\psi:&\mathcal{L}_c(F,G)\times\mathcal{L}_c(E,F)&\to&\mathcal{L}_c(E,G)\\
		&	(v,u)								&\mapsto& v\circ u\\
	\end{array}\]
	Let $k\in [\![1,n-1]\!]$. Differentiating $(g\circ f)^1$ formally $k$ times gives the following formula for $x\in U$ and $v\in E^{\otimes k}$:
	\begin{equation}\label{eq:BiDerivativeForCompo}
	((g\circ f)^1)^k(x)(v)=\underset{\sigma\in\mathcal{S}_k}{\underset{i\in[\![0,k]\!]}\sum}{\frac{\psi((g^1\circ f)^i(x),(f^1)^{k-i}(x))}{i!(k-i)!}\sigma(v)}
	\end{equation}	 
	For $i\in[\![1,k]\!]$, the $i$\textsuperscript{th} chain rule derivative of $g^1\circ f$ defines, for $x\in U$ and $w\in E^{\otimes i}$, $(g^1\circ f)^i(x)(w)$ as the sum:
	\begin{equation}\label{eq:CompoDerivativeForCompo}
	(g^1\circ f)^i(x)(w)=\sum\limits_{j=1}^i \frac{(g^1)^j(f(x))}{j!} \underset{\underset{m_1+\cdots+m_j=i}{1\leq m_1,\ldots,m_j\leq n}}\sum \frac{f^{m_1}(x)\otimes \cdots\otimes f^{m_j}(x)}{m_1!\cdots m_j!}(\sum\limits_{\tau\in\mathcal{S}_i}(\tau(w)))
	\end{equation}	
	For $x\in U$ and $v_1,\ldots v_k,v_{k+1}\in E$  (we denote $v=v_1\otimes\cdots \otimes v_k$), equations (\ref{eq:BiDerivativeForCompo}) and (\ref{eq:CompoDerivativeForCompo}) identifiy  $((g\circ f)^1)^k(x)(v)(v_{k+1})$ as the sum:
	\[\begin{array}{l}
	(g^1(f(x))(f^{k+1}(x)(v\otimes v_{k+1}))+\\
	 \sum\limits_{1\leq j \leq i \leq k}
	\frac{g^{j+1}(f(x))}{j!i!(k-i)!} 
	\underset{\underset{m_1+\cdots+m_j=i}{1\leq m_1,\ldots,m_j\leq n}}\sum 
	\frac{f^{m_1}(x)\otimes \cdots\otimes f^{m_j}(x)\otimes f^{k-i+1}(x)}{m_1!\cdots m_j!}
	\left( {\underset{\tau\in\mathcal{S}_{\sigma(1),\ldots,\sigma(i)}}{\sum\limits_{\sigma\in\mathcal{S}_k}}}
	\tau(\sigma^{1,i}(v))\otimes \sigma^{2,i}(v)\otimes v_{k+1} \right) \\
	\end{array}\]
	which reads by re-indexing its terms:
	\[\begin{array}{l}
	(g^1(f(x))(f^{k+1}(x)(v\otimes v_{k+1}))+\\
	\underset{\underset{p-1\leq i \leq k}{2\leq p \leq k+1}}\sum 
	\frac{g^{p}(f(x))}{(p-1)!i!(k-i)!} 
	\underset{\underset{m_1+\cdots+m_{p-1}=i}{1\leq m_1,\ldots,m_{p-1}\leq n}}\sum 
	\frac{f^{m_1}(x)\otimes \cdots\otimes f^{m_{p-1}}(x)\otimes f^{k-i+1}(x)}{m_1!\cdots m_{p-1}!}\left( {\underset{\tau\in\mathcal{S}_{\sigma(1),\ldots,\sigma(i)}}{\sum\limits_{\sigma\in\mathcal{S}_k}}}
	\tau(\sigma^{1,i}(v))\otimes \sigma^{2,i}(v)\otimes v_{k+1} \right) \\
	\end{array}\]	
	
	To finish the identification between $((g\circ f)^1)^k(x)$ and $(g\circ f)^{k+1}(x)$, it is sufficient to prove the following identity, for $p\in[\![2,k+1]\!]$:
	\[\begin{array}{ll}
	\frac{g^{p}(f(x))}{p!}\underset{\underset{m_1+\cdots+m_p=k+1}{1\leq m_1,\ldots,m_p\leq n+1}}\sum 
	\frac{f^{m_1}(x)\otimes \cdots\otimes f^{m_p}(x)}{m_1!\cdots m_p!}\sum\limits_{\sigma\in\mathcal{S}_{k+1}}v_{\sigma(1)}\otimes\cdots\otimes v_{\sigma(k+1)}= 	\\
	\frac{g^{p}(f(x))}{(p-1)!} 
	\sum\limits_{i=p-1}^k
	\underset{\underset{m_1+\cdots+m_{p-1}=i}{1\leq m_1,\ldots,m_{p-1}\leq n}}\sum 
	\frac{f^{m_1}(x)\otimes \cdots\otimes f^{m_{p-1}}(x)\otimes f^{k-i+1}(x)}{m_1!\cdots m_{p-1}!i!(k-i)!}&\\
	{\underset{\tau\in\mathcal{S}_{\sigma(1),\ldots,\sigma(i)}}{\sum\limits_{\sigma\in\mathcal{S}_k}}}v_{\tau(\sigma(1))}\otimes\cdots\otimes v_{\tau(\sigma(i))}\otimes v_{\sigma(i+1)}\otimes\cdots\otimes v_{\sigma(k)}\otimes v_{k+1}&\\
	\end{array}\]
	Which is a straightforward consequence of the lemmas \ref{SigmaReasoning1} and \ref{SigmaReasoning2} combined with the symmetric property of $g^{p}$.
	\end{proof}
	%%%%%%
	% General case
	%%%%%%
	\subsubsection{General case}
	\textbf{In the remainder of this paper}, we will assume that the base space (almost always denoted $E$ below) and its successive tensors are endowed with norms satisfying the following property: 
	\[\forall k , \forall v\in E^{\otimes k}, \forall \tilde{v}\in E: \quad \|v \otimes \tilde{v}\|= \|v\|\|\tilde{v}\|\]
	Following remark \ref{FiniteDimIsometryTensorNorm}, this can always be satisfied if $E$ is finite dimensional. The results of this section remain true however if looser conditions on these norms (as in lemma \ref{RecursiveLipCtrl}) hold. We now proceed to showing that the image of a product of Lipschitz maps by a bilinear map is also Lipschitz:	
	%%%%
	%%%% COMPOSITION WITH BILINEAR FUNCTIONS
	%%%%
	\begin{Prop}\label{Compobifunc} Let $E$, $F$, $G$ and $H$ be normed vector spaces and $U$ be a subset of $E$. Let $\gamma>0$ and let $f:U \to F$ and $g:U \to G$ be two $\textrm{Lip}-\gamma$ maps. Let $B: F\times G\to H$ a continuous bilinear map. We assume that $(E^{\otimes k})_{k\geq 1}$ are endowed with norms satisfying the projective property. Then $B(f,g):U \to H$ (when endowed with its bilinear derivatives up to order $\lfloor \gamma \rfloor$) is $\textrm{Lip}-\gamma$ and there exists a constant $C$ (depending only on $\gamma$) such that:
	\[\|B(f,g)\|_{\textrm{Lip}-\gamma} \leq C\|B\|\|f\|_{\textrm{Lip}-\gamma}\|g\|_{\textrm{Lip}-\gamma}\]
	\end{Prop}
	 The idea behind the proof is rather simple but contains notions and ideas that will be very important to the proof of the main theorem of this section.
	\begin{proof} Let $\varepsilon\in (0,1]$. We prove our statement by induction on $n$ $(n=\lfloor \gamma \rfloor, \gamma= n + \varepsilon)$. For $n=0$, the proof of the statement is trivial and is left as an exercise. Let $n\in\mathbb{N}$. We assume the statement true for $n$ and let us prove it for $n+1$.\\
	Let $E$, $F$, $G$ and $H$ be normed vector spaces and $U$ be a subset of $E$. Let $f:U \to F$ and $g:U \to G$ be two $\textrm{Lip}-(n+1+\varepsilon)$ maps and $B: F\times G\to H$ be a continuous bilinear map. We will show that $(Z,Z^1,\ldots,Z^{n+1})$ is $\textrm{Lip}-(n+1+\varepsilon)$ where, $Z:=B(f,g)$ and for $k\in[\![1,n+1]\!]$, $Z^k$ is the $k$\textsuperscript{th} bilinear derivative of $Z$.\\
	First, we prove that $Z^1$ is $\textrm{Lip}-(n+\varepsilon)$ and with a well bounded $\textrm{Lip}-(n+\varepsilon)$ norm. As $f^1$ and $g$ (resp. $f$ and $g^1$) are both $\textrm{Lip}-(n+\varepsilon)$, then, by the induction hypothesis, $B(f^1,g)$ (resp. $B(f,g^1)$) is $\textrm{Lip}-(n+\varepsilon)$ when endowed with its bilinear derivatives. Hence, $(Z^1,(Z^1)^1,\ldots,(Z^1)^n)$ is $\textrm{Lip}-(n+\varepsilon)$ and (using the induction hypothesis) there exists a constant $c_{n,\varepsilon}$ such that:
	\[\|Z^1\|_{\textrm{Lip}-(n+\varepsilon)} \leq c_{n,\varepsilon}\|B\|\|f\|_{\textrm{Lip}-(n+1+\varepsilon)}\|g\|_{\textrm{Lip}-(n+1+\varepsilon)}\]
	Moreover, by proposition \ref{IDFormalDeriBilinear}, for all $k\in[\![0,n]\!]$: $(Z^1)^k=Z^{k+1}$.\\
	For $k\in[\![0,n+1]\!]$, let $R_k$ (resp. $S_k$) be the remainder of order $k$ associated to $f$ (resp. $g$). Let $z\in U$ and let $x,y\in B(z,1/2)\cap U$. Writing the Taylor expansion of $f$ and $g$ and using the bilinearity of $B$, we get:
	\[Z(x)=\sum\limits_{i=0}^{n+1}{Z^i(y)\left(\frac{(x-y)^{\otimes i}}{i!}\right)}+T_0(x,y)\]
	where:
	\[\begin{array}{rcl}
	T_0(x,y)&=&B(R_0(x,y),g(x))+B(f(x)-R_0(x,y),S_0(x,y))+ \\
		&&\underset{i+j>n+1}{\underset{i,j\in[\![0,n+1]\!]}\sum}{B(f^i(y),g^j(y))\left(\frac{(x-y)^{\otimes (i+j)}}{i!j!}\right)}\\
	\end{array}\]
	It is obvious that we can bound $Z(x)$ and $T_0(x,y)$ independently of $z$ and in the form suggested by the statement of the proposition. Hence, by lemma \ref{RecursiveLipCtrl}, $Z$ is $\textrm{Lip}-(n+\varepsilon+1)$ over $B(z,1/2)\cap U$ and therefore, by lemma \ref{LocalLipGen}, it is $\textrm{Lip}-(n+\varepsilon+1)$ over $ U$ with the suggested control of the Lipschitz norm.
	\end{proof}
	
	\begin{Rem}
	By proposition \ref{Compobifunc}, the real-valued Lip-$\gamma$ functions form an algebra under point-wise multiplication.
	\end{Rem}
	
	\begin{Rem} If $E\otimes F$ is endowed with a norm satisfying the projective property, then the tensor product of an $E$-valued Lipschitz map by an $F$-valued Lipschitz map is also Lipschitz as a direct consequence of proposition \ref{Compobifunc}.
	\end{Rem}
	
	Using proposition \ref{Compolinfunc2} for example, one may argue that we don't need both maps $u$ and $f$ to be Lipschitz in order for their composition to be Lipschitz too. In particular, we don't need the map $u$ in said proposition to be bounded. This has a more important consequence than one might think at first: if $u$ and its derivatives vary slowly compared to $\|u\|_{\infty}$, then the bounds we obtain on $\|f\circ u\|_{\textrm{Lip}-\gamma}$ are not as precise as one might wish if we only use the quantity $\|u\|_{\textrm{Lip}-\gamma}$ as the latter includes information about $\|u\|_{\infty}$ (this being unnecessary for the whole proof to work). This is one of the reasons for which we introduce almost Lipschitz maps. This notion will be of use in future work but we will already see it in use in two contexts in the remainder of this paper: the study of smoothness of flows of Lipschitz vector fields and the inverse function theorem in the Lipschitz case; the reason being that they provide a sharper control and description of a map's behaviour.
	%%%%
	%%%% Almost lipschitz of order 1
	%%%%
	\begin{Def} Let $n\in \mathbb{N}^*$, $0< \varepsilon\leq 1$ and $\delta \in (0,\infty)\cup \{\infty\}$. Let $E$ and $F$ be two normed vector spaces and $U$ be a subset of $E$. For every $k \in [\![0,n]\!]$, let $f^k:U\to \mathcal{L}_s(E^{\otimes k},F)$ be a map with values in the space of the symmetric $k$-linear mappings from $E$ to $F$. Denote by $R_0:U\times U\to E$ the remainder map of order $0$ associated to $f=(f^0,f^1,\ldots,f^n)$.	The collection $f$ is said to be almost Lipschitz of degree $n+\varepsilon$ on domains of size $\delta$ of $U$ if $(f^1,\ldots,f^n)$ is $\textrm{Lip}-(n+\varepsilon-1)$ and there exists a non-negative constant $\tilde{M}$, such that:
	\[\forall x,y \in U:\quad \|x-y\|< \delta \Rightarrow \|R_0(x,y)\|\leq \tilde{M}\|x-y\|^{n+\varepsilon}\]
	If $\|R_0\|_{\infty,\delta}$ denotes the smallest value for such a constant $\tilde{M}$, we will denote:
	\[\|f\|_{\delta,\textrm{Lip}-(n+\varepsilon)}=\max(\|f^1\|_{\textrm{Lip}-(n+\varepsilon-1)} , \|R_0\|_{\infty,\delta})\]
	When $\delta$ is infinite or its value is irrelevant, we will merely say that $f$ is almost Lip-$(n+\varepsilon)$ on $U$.
	\end{Def}
	\begin{Rem}	$\|.\|_{\delta,\textrm{Lip}-(n+\varepsilon)}$ does not define a norm as it vanishes for all constant maps.
	\end{Rem}
	%%%%
	%%%% EXAMPLE: LINEAR MAPS ARE ALMOST LIPSCHITZ 
	%%%%
	\begin{example} Let $\gamma>1$ and $E$ and $F$ be two normed vector spaces. Let $u: E\to F$ be a bounded linear map. Then $(u,x\mapsto u,0,\ldots,0)$ is almost Lipschitz of degree $\gamma$ on $E$. Moreover $\|u\|_{\infty,\textrm{Lip}-\gamma}=\|u\|$. \end{example}
	The following lemma can be seen as a reformulation of lemma \ref{RecursiveLipCtrl} combined, if necessary, with the statement of lemma \ref{LocalLipGen}.
	%%%%
	%%%% LEMMA: LIPSCHITZ MAPS ARE ALMOST LIPSCHITZ 
	%%%%
	\begin{lemma} Let $\gamma>1$ and $\delta \in (0,\infty)\cup \{\infty\}$. Let $E$ and $F$ be two normed vector spaces and $U$ be a subset of $E$. We assume that $(E^{\otimes k})_{1\leq k \leq \lfloor\gamma\rfloor}$ are endowed with norms satisfying the projective property. Then a map $f: U \rightarrow F$ is $\textrm{Lip}-\gamma$ if and only if $f$ is bounded and is almost Lipschitz of degree $\gamma$ on domains of size $\delta$ of $U$. In this case:
	\[\|f\|_{\textrm{Lip}-\gamma}=\max(\|f\|_{\infty},\|f\|_{\delta,\textrm{Lip}-\gamma})\]
	\end{lemma}
	On open convex sets, smooth maps are almost Lipschitz if and only if their derivatives are Lipschitz:
	%%%%
	%%%% Diff Map <=> Almost LIP 
	%%%%
	\begin{Prop}\label{AlmostLipDiffMaps} Let $n\in \mathbb{N}^*$, $0< \varepsilon\leq 1$. Let $E$ and $F$ be two normed vector spaces and $U$ be an open convex subset of $E$.  We assume that $(E^{\otimes k})_{1\leq k \leq n}$ are endowed with norms satisfying the projective property. Let $f:U\to F$ be a map of class $\mathcal{C}^n$ with successive derivatives respectively denoted $f^1,\ldots,f^n$. Then $f$ is almost Lipschitz of degree $n+\varepsilon$ on $U$ if and only $f^1$ is Lipschitz of degree $n+\varepsilon-1$ on $U$. In this case:
	\[\forall \delta>0: \quad \|f\|_{\delta,\textrm{Lip}-(n+\varepsilon)}=\|f^1\|_{\textrm{Lip}-(n+\varepsilon-1)}\]
	\end{Prop}
	\begin{proof}
	By definition, if $f$ is almost Lip-$(n+\varepsilon)$ then $f^1$ is Lip-$(n+\varepsilon-1)$ on $U$. Assume now that $f^1$ is Lip-$(n+\varepsilon-1)$ on $U$. Then (theorem \ref{CaracLip}) $f^n$ is $\varepsilon$-H\"older. Using the Taylor expansion with integral remainder of $f$, we get the required upper-bound on its remainder map of order $0$. 
	\end{proof}
	%%%%
	%%%% Almost Lip => 1-Holder
	%%%%
	\begin{lemma}\label{AlmostLipIsHolder} Let $n\in \mathbb{N}^*$, $0< \varepsilon\leq 1$ and $\delta \in (0,\infty)$. Let $E$ and $F$ be two normed vector spaces and $U$ be a subset of $E$.  If $n>1$, we assume that $(E^{\otimes k})_{1\leq k \leq n}$ are endowed with norms satisfying the projective property. Let $f=(f^0,f^1,\ldots,f^n)$ be an almost Lip-$(n+\varepsilon)$ map on domains of size $\delta$ of $U$. Then there exists a constant $C_{n,\varepsilon,\delta}$ depending only on $n$, $\varepsilon$ and $\delta$ such that: 
	\[\forall x,y \in U:\quad \|x-y\|< \delta \Rightarrow \|f^0(x)-f^0(y)\|\leq C_{n,\varepsilon,\delta} \|f\|_{\delta,\textrm{Lip}-(n+\varepsilon)} \|x-y\|\]
	\end{lemma}
	\begin{proof}
	Let $x,y \in U$ such that $ \|x-y\|< \delta$. From the Taylor-like expansion of $f^0$, we get:
	\[\|f^0(x)-f^0(y)\|\leq  \sum\limits_{j=1}^{n} \|f^1\|_{\textrm{Lip}-(n+\varepsilon-1)}\frac{\|x-y\|^{j}}{j!}+ \|R_0\|_{\infty,\delta}\|x-y\|^{n+\varepsilon} \]
	By discussing for example the cases whether $\delta$ is larger than 1, we get:
	\[\|f^0(x)-f^0(y)\|\leq \|f\|_{\delta,\textrm{Lip}-(n+\varepsilon)}\|x-y\| e \max{(1,\delta^{n+\varepsilon-1})}\]
	\end{proof}
	We start by showing that the composition of a Lipschitz map with an almost Lipschitz map (both of degree $1+\varepsilon$) is $\textrm{Lip}-(1+\varepsilon)$.
	%%%%
	%%%% COMPOSITION Almost LIPSCHITZ 1+epsilon
	%%%%
	\begin{lemma}\label{CompolipfuncCASE21} Let $E$, $F$ and $G$ be three normed vector spaces. Let $U$ be a subset of $E$ and $V$ be a subset of $F$. Let $\varepsilon\in(0,1]$ and $\delta \in (0,\infty)$. Let $f=(f^0,f^1)$ be an almost Lipschitz map of degree $1+\varepsilon$ on domains of size $\delta$ of $U$ such that $f^0(U) \subseteq V$ and $g:V \to G$ be a $\textrm{Lip}-(1+\varepsilon)$ map.	Then $g \circ f$ is $\textrm{Lip}-(1+\varepsilon)$ (with $(g\circ f)^1$ defined as a formal 1\textsuperscript{st} chain rule derivative) and there exists a constant $M_{\varepsilon, \delta}$ (depending only on $\varepsilon$ and $\delta$) such that:
	\[\|g\circ f\|_{\textrm{Lip}-(1+\varepsilon)} \leq M_{\varepsilon,\delta}\|g\|_{\textrm{Lip}-(1+\varepsilon)} \max(1, \|f\|_{\delta,\textrm{Lip}-(1+\varepsilon)}^{1+\varepsilon})\]
	\end{lemma}
	\begin{proof} The following inequalities are straightforward: 
	\[\|(g \circ f)^0\|_{\infty}\leq \|g\|_{\textrm{Lip}-(1+\varepsilon)} \quad;\quad \|(g\circ f)^1 \|_{\infty}\leq \|g\|_{\textrm{Lip}-(1+\varepsilon)} \|f^1\|_{\textrm{Lip}-\varepsilon}\] 
	Denote by  $(R_0,R_1)$ and $(S_0,S_1)$ the associated remainders to $f=(f^0,f^1)$ and $g=(g^0,g^1)$ respectively. The remainders $T_0:U\times U\to G$ and $T_1:U\times U\to \mathcal{L}(E,G)$ associated to $g\circ f$ are given by:
	\[\begin{array}{lrcl}
	\forall x,y \in U:& T_0(x,y)&=&g^1(f^0(y))(R_0(x,y))+S_0(f^0(x),f^0(y)),\\
						 &T_1(x,y)&=&(g\circ f)^1(x)-(g\circ f)^1(y)\\
	\end{array}\]
	Let $C_{\varepsilon,\delta}$ be a constant (depending only on $\varepsilon$ and $\delta$ and chosen here to be larger than 1) such that: 
	\[\forall x,y \in U:\quad \|x-y\|< \delta \Rightarrow \|f^0(x)-f^0(y)\|\leq C_{ \varepsilon,\delta} \|f\|_{\delta,\textrm{Lip}-(n+\varepsilon)} \|x-y\|\]
	Let $x,y \in U$ such that $ \|x-y\|< \delta$. We have:
	\[\|g^1(f^0(y))(R_0(x,y))\|\leq\|g\|_{\textrm{Lip}-(1+\varepsilon)}\|R_0\|_{\infty,\delta}\|x-y\|^{1+\varepsilon}\]
	and:
	\[\|S_0(f^0(x),f^0(y))\|\leq\|g\|_{\textrm{Lip}-(1+\varepsilon)}C_{ \varepsilon,\delta}^{1+\varepsilon}\ \|f\|_{\delta,\textrm{Lip}-(n+\varepsilon)}^{1+\varepsilon}\|x-y\|^{1+\varepsilon}	\]
	Hence:
	\[\|T_0(x,y)\|\leq\|g\|_{\textrm{Lip}-(1+\varepsilon)}\max(1, \|f\|_{\delta,\textrm{Lip}-(1+\varepsilon)}^{1+\varepsilon}) (1+C_{ \varepsilon,\delta}^{1+\varepsilon})\|x-y\|^{1+\varepsilon}\]
	With using similar techniques as above, one gets the inequality:
	\[\|T_1(x,y)\|\leq\|g\|_{\textrm{Lip}-(1+\varepsilon)}\max(1, \|f\|_{\delta,\textrm{Lip}-(1+\varepsilon)}^{1+\varepsilon}) (1+C_{ \varepsilon,\delta}^{\varepsilon})\|x-y\|^{\varepsilon}\]
	Therefore $g\circ f$ is $\textrm{Lip}-(1+\varepsilon)$ on the intersection of $U$ with balls of radius $\delta/2$ with a norm upper-bounded by:
	\[2  C_{ \varepsilon,\delta}^{1+\varepsilon} \|g\|_{\textrm{Lip}-(1+\varepsilon)}\max(1, \|f\|_{\delta,\textrm{Lip}-(1+\varepsilon)}^{1+\varepsilon})\]
	We conclude by using lemma \ref{LocalLipGen}.
	\end{proof}
	Now we get to the main theorem of this section:
	%%%%
	%%%% COMPOSITION Almost LIPSCHITZ GENERAL
	%%%%
	\begin{theo}\label{CompolipfuncCASE22} Let $E$, $F$ and $G$ be three normed vector spaces. Let $U$ be a subset of $E$ and $V$ be a subset of $F$. Let $\delta>0$ and $\gamma\geq1$. We assume that $(E^{\otimes k})_{k\geq1}$ and $(F^{\otimes k})_{k\geq1}$ are endowed with norms satisfying the projective property. Let $f:U \to F$ be an almost $\textrm{Lip}-\gamma$ map on domains of size $\delta$ of $U$ and $g:V \to G$ be a $\textrm{Lip}-\gamma$ map such that $f(U) \subseteq V$. Then $g \circ f$ (defined via successive formal chain rule derivatives) is $\textrm{Lip}-\gamma$ and there exists a constant $C_{\gamma,\delta}$ (depending only on $\gamma$ and $\delta$) such that:
	\[\|g\circ f\|_{\textrm{Lip}-\gamma} \leq C_{\gamma, \delta}\|g\|_{\textrm{Lip}-\gamma}\max( \|f\|_{\delta,\textrm{Lip}-(n+\varepsilon)}^{\gamma},1)\]
	\end{theo}
	\begin{proof} We leave the case $\gamma=1$ as an easy and straightforward exercise. Let $\varepsilon \in (0,1]$. We will prove the following by induction: 
	\begin{Claim} For all $n\in\mathbb{N}^*$, for any normed vector spaces $E$, $F$, $G$ and $H$, such that $(E^{\otimes k})_{k\geq 1}$ and $(F^{\otimes k})_{k\geq 1}$ are endowed with norms satisfying the projective property, and any subsets $U$ of $E$ and $V$ of $F$, there exists a real constant $C_{n,\varepsilon,\delta}$ (depending only on $n$, $\varepsilon$ and $\delta$) such that if $f=(f^0,\ldots,f^n):U \to F$ is an almost $\textrm{Lip}-(n+\varepsilon)$ map on domains of size $\delta$ of $U$ and $g=(g^0,\ldots,g^n):V \to G$ is a $\textrm{Lip}-(n+\varepsilon)$ map such that $f(U) \subseteq V$, then $g \circ f$ is $\textrm{Lip}-(n+\varepsilon)$ and:
	\[\|g\circ f\|_{\textrm{Lip}-(n+\varepsilon)} \leq C_{n,\varepsilon, \delta} \|g\|_{\textrm{Lip}-(n+\varepsilon)}\max(\|f\|^{n+\varepsilon}_{\delta,\textrm{Lip}-(n+\varepsilon)},1)\]
	\end{Claim}
	
	% Induction: n=1
	The case $n=1$ has been proved in lemma \ref{CompolipfuncCASE21}.\\	
	% General induction
	Let now $n\in\mathbb{N}^*$. We assume that the assertion is true for $n$ and let us prove it for $n+1$. Let $f=(f^0,\ldots,f^{n+1}):U \to F$ be an almost $\textrm{Lip}-(n+1+\varepsilon)$ on domains of size $\delta$ of $U$ and $g=(g^0,\ldots,g^{n+1}):V \to G$ be a $\textrm{Lip}-(n+1+\varepsilon)$ map such that $f(U) \subseteq V$, and with remainders denoted by $R_0, \ldots, R_{n+1}$ and $S_0, \ldots, S_{n+1}$ respectively. Let $x,y\in U$. Define $P_0(x,y)=R_0(x,y)$, and for every $k\in [\![1,n+1]\!]$: $P_k(x,y)=f^k(y)\frac{(x-y)^{\otimes k}}{k!}$ and finally:
	\[\begin{array}{rcl}
	T_0(x,y)&=&S_0(f^0(x),f^0(y))+ \sum\limits_{j=1}^{n+1} \frac{g^j(f^0(y))}{j!}\left( \underset{\underset{i_1\cdots i_j=0}{0\leq i_1,\ldots,i_j\leq n+1}}{\sum} P_{i_1}(x,y)\otimes\cdots\otimes P_{i_j}(x,y)+\right.\\
						&&\left. \underset{\underset{i_1+\cdots+i_j>n+1}{1\leq i_1,\ldots,i_j\leq n+1}}{\sum}f^{i_1}(y)\otimes\cdots\otimes f^{i_j}(y)\frac{(x-y)^{\otimes (i_1+\cdots+i_j)}}{i_1!\cdots i_j!}\right)\\
	\end{array}\]
	Then, we can simply write:
	\[\begin{array}{rcl}
	g^0(f^0(x))
	&=&g^0(f^0(y))+\sum\limits_{k=1}^{n+1} (g\circ f)^k(y)\displaystyle\frac{(x-y)^{\otimes k}}{k!}+T_0(x,y)\\
	\end{array}\]
	Assume that $\|x-y\|<\delta$. It is an easy exercise then to show, using lemma \ref{AlmostLipIsHolder}, that there exists a constant $M_{n,\varepsilon,\delta}$ (depending only on $n$, $\varepsilon$ and $\delta$) such that:
	\[\|T_0(x,y)\|\leq  M_{n,\varepsilon,\delta} \|g\|_{\textrm{Lip}-(n+1+\varepsilon)}\max(\|f\|^{n+1+\varepsilon}_{\delta, \textrm{Lip}-(n+1+\varepsilon)},1) \|x-y\|^{n+1+\varepsilon}\]
	We have also that $\|g^0\circ f^0\|_{\infty}\leq\|g\|_{\textrm{Lip}-(n+1+\varepsilon)}$. All that remains to do then to end the proof is to show that $((g\circ f)^1,\ldots,((g\circ f)^1)^{n})$ is $\textrm{Lip}-(n+\varepsilon)$ on the intersection of every ball of radius $\delta/2$ with $U$ (with a uniformly well controlled norm) then conclude using lemma \ref{HighDerivAgree}, lemma \ref{RecursiveLipCtrl} and then finally lemma \ref{LocalLipGen}. As $g^1$ is $\textrm{Lip}-(n+\varepsilon)$ and $f^0$ is almost $\textrm{Lip}-(n+\varepsilon)$, then $g^1 \circ f^0$  is $\textrm{Lip}-(n+\varepsilon)$ by the induction hypothesis. More precisely, there exists a constant $m$ (depending only on  $n$ and $\varepsilon$) such that:
	\[\|g^1 \circ f^0\|_{\textrm{Lip}-(n+\varepsilon)} \leq m \|g\|_{\textrm{Lip}-(n+\varepsilon+1)}\max(\|f\|^{n+\varepsilon}_{\delta, \textrm{Lip}-(n+\varepsilon+1)},1)\]
	Define:
	\[\begin{array}{rccc}
	\psi:&\mathcal{L}_c(F,G)\times\mathcal{L}_c(E,F)&\to&\mathcal{L}_c(E,G)\\
		&	(v,u)								&\mapsto& v\circ u\\
	\end{array}\]
	$\psi$ is a continuous bilinear map with norm $1$.  As $(g\circ f)^1=\psi(g^1 \circ f^0,f^1)$ then, using proposition \ref{Compobifunc}, $((g\circ f)^1,\ldots,((g\circ f)^1)^n)$ is $\textrm{Lip}-(n+\varepsilon)$ and there exists a constant $C_{n,\varepsilon}$ depending only on $n$ and $\varepsilon$ such that: 
	\[\|(g\circ f)^1\|_{\textrm{Lip}-(n+\varepsilon)} \leq C_{n,\varepsilon}\|g\|_{\textrm{Lip}-(n+\varepsilon+1)}\max(\|f\|^{n+\varepsilon+1}_{\delta, \textrm{Lip}-(n+\varepsilon+1)},1)\] 
	which concludes the induction argument.
	 \end{proof}
	 As a corollary of theorem \ref{CompolipfuncCASE22}, we can claim that a well-defined composition of Lipschitz maps is itself Lipschitz. This particular result has already appeared in \cite{TCL} but with a slight mistake in the control of the Lipschitz norm of the composition map that we correct here (this goes along with the full and detailed proof of theorem \ref{CompolipfuncCASE22} that differs from the one suggested in the aforementioned paper):
	%%%%
	%%%% COMPOSITION LIPSCHITZNESS GENERAL
	%%%%
	\begin{theo}\label{Compolipfunc} Let $E$, $F$ and $G$ be three normed vector spaces. Let $U$ be a subset of $E$ and $V$ be a subset of $F$. Let $\gamma>1$. We assume that $(E^{\otimes k})_{k\geq1}$ and $(F^{\otimes k})_{k\geq1}$ are endowed with norms satisfying the projective property. Let $f:U \to F$ and $g:V \to G$ be two $\textrm{Lip}-\gamma$ maps such that $f(U) \subseteq V$. Then $g \circ f$ (defined via successive formal chain rule derivatives) is $\textrm{Lip}-\gamma$ and there exists a constant $C_{\gamma}$ (depending only on $\gamma$) such that:
	\[\|g\circ f\|_{\textrm{Lip}-\gamma} \leq C_{\gamma}\|g\|_{\textrm{Lip}-\gamma}\max(\|f\|^{\gamma}_{\textrm{Lip}-\gamma},1)\]
	\end{theo}	
	\begin{Rem} We can obtain an easier proof for theorem \ref{Compolipfunc} by using the extension theorems that we will review in subsection \ref{sec:ExtensionThm}. The inequality will still prove hard to get and will involve a constant depending on the dimension of the spaces, an inconvenience that we don't have in the proof presented above.
	\end{Rem}
	\begin{Rem}
	Using theorem \ref{CompolipfuncCASE22}, we can now retrieve (up to a multiplicative constant) the result of proposition \ref{Compolinfunc2} by treating linear maps as almost Lipschitz maps.
	\end{Rem}

	%%
	%A quantitative estimate
	%%
	\subsection{A quantitative estimate}
	In this section, we give some more precise local quantitative estimates (in the Lipschitz norm) if the value of a Lipschitz map at a point is known.
	
	%%%% FUNCTIONS AGREEING ON A POINT %%%%
	\begin{theo}\label{QuantEstimBasic} Let $\gamma,\gamma'>0$ such that $\gamma'<\gamma$. Let $E$ and $F$ be two normed vector spaces, $U$ be a subset of $E$ and $x_0\in U$. Let $f=(f^0,\ldots,f^{\lfloor\gamma\rfloor})$ be a $\textrm{Lip}-\gamma$ map on $U$ with values in $F$. Assume that for all $k\in [\![0,{\lfloor\gamma\rfloor}]\!]: f^k(x_0)=0$. We also assume that $(E^{\otimes k})_{1\leq k \leq {\lfloor\gamma\rfloor}}$ are endowed with norms satisfying the projective property. Then there exists a constant $C_{\gamma, \gamma'}$ (that can be chosen to depend only, and continuously, on the difference $\gamma- \gamma'$) such that for all $\delta>0$ one has:
	\[\|f\|_{\textrm{Lip}-\gamma', B(x_0,\delta)\cap U}\leq C_{\gamma, \gamma'} \|f\|_{\textrm{Lip}-\gamma}\delta^{\gamma-\gamma'} \]
	\end{theo}
	\begin{proof} Let $n,n'\in\mathbb{N}$ and $\varepsilon,\varepsilon'\in(0,1]$ such that $\gamma=n+\varepsilon$ and $\gamma'=n'+\varepsilon'$. Denote by $(R_0,\ldots,R_n)$ the remainders associated to $f$. Let $\delta>0$ and let $k\in[\![0,n]\!]$.
	Let $x\in B(x_0,\delta)\cap U$ and $v\in E^{\otimes k}$. Then, as $f$ is $\textrm{Lip}-\gamma$ and that $f^j(x_0)=0$ for all $j\in [\![0,n]\!]$, we get:
	\[\begin{array}{rcl}
	\|f^k(x)(v)\|&=&\|\sum\limits_{k}^{n} f^j(x_0)(\frac{v\otimes (x-x_0)^{\otimes (j-k)}}{(j-k)!})+R_k(x,x_0)(v)\|\\
			&=&\|R_k(x,x_0)(v)\|\\
			&\leq&\|f\|_{\textrm{Lip}-\gamma}\|x-x_0\|^{\gamma-k}\|v\|\\
			&\leq&\|f\|_{\textrm{Lip}-\gamma}\delta^{\gamma-k}\|v\|\\
	\end{array}\]
	Therefore, $\sup\limits_{x\in B(x_0,\delta)\cap U}\|f^k(x)\|\leq\|f\|_{\textrm{Lip}-\gamma}\delta^{\gamma-k}$, for all $k\in[\![0,n]\!]$. In particular:
	\[\max\limits_{0\leq k \leq n' }\|f^k\|_{\infty, B(x_0,\delta)\cap U}\leq\|f\|_{\textrm{Lip}-\gamma}\min(1,\delta^{\gamma-n'})\]
	Let $k\in[\![0,n']\!]$. We define $S_k:U\times U\to\mathcal{L}(E^{\otimes k},F)$ (the new remainder) by:
	\[S_k(x,y)(v)=f^k(x)(v)-\sum\limits_{j=k}^{n'} f^j(y)(\frac{v\otimes (x-y)^{\otimes (j-k)}}{(j-k)!})\]
	Let $x,y\in B(x_0,\delta)\cap U$ and $v\in E^{\otimes k}$. Writing the Taylor expansion of $f$ as a $\textrm{Lip}-\gamma$ map, we get the following identity:
	\[S_k(x,y)(v)=\sum\limits_{j=n'+1}^{n} f^j(y)(\frac{v\otimes (x-y)^{\otimes (j-k)}}{(j-k)!})+R_k(x,y)(v)\]
	which, using our new upper-bound for $\|f^k\|_{\infty,B(x_0,\delta)\cap U}$, leads to the inequality:
	\[\begin{array}{rcl}
	\|S_k(x,y)(v)\|&\leq&\|f\|_{\textrm{Lip}-\gamma}\left(\sum\limits_{j=n'+1}^{n} \delta^{\gamma-j}\frac{\|x-y\|^{j-k}}{(j-k)!}+\|x-y\|^{\gamma-k}\right)\|v\|\\
				&\leq&\|f\|_{\textrm{Lip}-\gamma}\|x-y\|^{\gamma'-k}\left(\sum\limits_{j=n'+1}^{n} \delta^{\gamma-j}\frac{\|x-y\|^{j-\gamma'}}{(j-k)!}+\|x-y\|^{\gamma-\gamma'}\right)\|v\|
	\end{array}\]
	Therefore:
	\[\underset{x\neq y}{\sup\limits_{x,y\in B(x_0,\delta)\cap U}}\frac{\|S_k(x,y)\|}{\|x-y\|^{\gamma'-k}} \leq\|f\|_{\textrm{Lip}-\gamma}\delta^{\gamma-\gamma'}\left(
\sum\limits_{j=\lfloor\gamma'\rfloor+1}^{\lfloor\gamma\rfloor}\frac{2^{j-\gamma'}}{(j-\lfloor\gamma'\rfloor)!}+2^{\gamma-\gamma'} 
\right)\]
	By taking for example $C_{\gamma, \gamma'}=e^2+2^{\gamma-\gamma'}$, we get the sought result.\end{proof}
	\begin{Rem} With the notations of the previous theorem, if we only have $f^k(x_0)=0$ for $k\in [\![0,n']\!]$, the result remains essentially true but with a slightly different upper-bound (that still converges to $0$ as $\delta$ goes to $0$). However, in the cases where $\gamma'=\gamma$ or if there exists $k\in [\![0,n']\!]$ such that $f^k(x_0)\neq0$ then we cannot get a better control of $\|f\|_{\textrm{Lip}-\gamma', B(x_0,\delta)\cap U}$ than $\|f\|_{\textrm{Lip}-\gamma}$ as the example below shows (we can, nevertheless, improve the control of $\|f^{k}\|_{\infty,B(x_0,\delta)\cap U}$ for all $k\in [\![0,n]\!]$ in the first case).
	\end{Rem}

	\begin{example}
	Consider the function $f: x\mapsto x$ defined on $(-1,1)$. As $f$ is smooth, $f$ is Lipschitz of any degree. \begin{itemize}
	\item (Case $\gamma'=\gamma=1$) On the one hand, $f$ is $\textrm{Lip}-1$, $\|f\|_{\textrm{Lip}-1}=1$ and $f(0)=0$. On the other hand, for any $\delta \in (0,1]$, there does not exist a constant $\lambda$ strictly less than $1$ such that:
	\[\forall x,y \in (\delta,-\delta):\quad |f(x)-f(y)|\leq\lambda|x-y|\]
	Therefore : 
	\[\forall \delta \in (0,1]:\quad \|f\|_{\textrm{Lip}-1, (\delta,-\delta)}=1\]
	However, we still have $\|f\|_{\infty, (\delta,-\delta)}\underset{\delta\rightarrow 0}\longrightarrow 0$
	\item (Case $\gamma'=3/2, \gamma=2$) $f$ is $\textrm{Lip}-2$, $\|f\|_{\textrm{Lip}-2}=1$ and $f(0)=0$. However, $f'(0)\neq 0$ and:
	\[\forall \delta \in (0,1]:\quad \|f\|_{\textrm{Lip}-3/2, (\delta,-\delta)}=1\]
	\end{itemize}
	\end{example}

	\begin{Rem} Using theorem \ref{QuantEstimBasic}, one can easily then compare two Lip-$\gamma$ maps in the Lip-$\gamma'$ norm, when $\gamma'<\gamma$, which values and ``successive derivatives"' values (in the sense of a Lipschitz map) agree at one point. \end{Rem}

	%%
	% Whitney's extension theorem
	%%
	\subsection{Extension theorems and a short review of the literature}\label{sec:ExtensionThm}
	One of the most interesting and still open problems in Lipschitz geometry, and classical analysis in general, is about the existence of extensions of Lipschitz (or smooth) maps to the whole space and the control of the Lipschitz norm of the extension. This is known as Whitney's extension problem and can be informally stated in the following way:
	\begin{quote}
	Given an arbitrary set $A$ and a map $f: A (\subseteq E)\rightarrow F$, where $E$ and $F$ are vector spaces:
\begin{enumerate}\item In which ways can one define $f$ to be a \emph{smooth} map on $A$ so that, if $\mathring{A}$ (the interior of $A$) is not empty, $f_{|\mathring{A}}$ is smooth in the classical sense?
\item Given such a definition, can we extend $f$ to the whole space $E$ so that this extension is smooth in the classical sense?
\end{enumerate}
\end{quote}

	Since Whitney introduced it in a series of three seminal papers \cite{Whitney1, Whitney2, Whitney3}, several mathematicians have been working on this problem, mostly in the case where both $E$ and $F$ are finite dimensional. Whitney himself was the first one to suggest a solution in the case where $A$ is the closure of a region. The answers to this question are of crucial importance. For instance, the theory of rough paths (as presented in \cite{Flour} and \cite{Revista}) requires that the vector fields appearing in a rough differential equation be Lipschitz (this restriction enables one to derive global solutions along with the rate of convergence of the Picard iterations). As linear and polynomial functions in particular are not in this class of functions, using Whitney's theorem allows to extend the restriction of polynomials to compact sets (which are Lipschitz) to the whole space in a way that they stay Lipschitz. Another illustration would be the construction of a suitable function from sampled data so that one can work in the appropriate class of functions associated to the experiment's physical model.\\
	
	We will state below two other examples of such results: one in which one can extend Lipschitz maps of any degree to the whole space, but at the cost of amplifying the Lipschitz norm; and another one where the extension has the same Lipschitz norm as the map we start with but which is currently only obtained for Lipschitz-$1$ maps in the framework of Hilbert spaces.\\

	Stein discussed in some length the class of Lipschitz functions in \cite{Stein} and gave a Whitney extension theorem in this case: 
	\begin{theo}[Stein \cite{Stein}] Let $\gamma\geq1$. Let $E$ and $F$ be two finite dimensional vector spaces and $K$ be a closed subset of $E$. There exists a continuous linear map sending every $F$-valued Lip-$\gamma$ map $f$ defined on $K$ to an $F$-valued Lip-$\gamma$ map $\tilde{f}$ defined on $E$ such that $\tilde{f}_{|K}=f$. Moreover, the norm of the linear extension map depends only on $\gamma$ and the dimensions of $E$ and $F$.
	\end{theo}

	\begin{theo}[Kirszbraun \cite{Kirs}] Let $H_1$ and $H_2$ be two Hilbert spaces. Let $A$ be a subset of $H_1$, $K\geq0$ and $f:A\rightarrow H_2$ be a map such that:
	\[\forall x,y\in A:\quad \|f(x)-f(y)\|\leq K\|x-y\|\]
	(i.e. $f$ is $1$-H\"older). Then there exists a map $\tilde{f}:H_1\rightarrow H_2$ such that $\tilde{f}_{|A}=f$ and:
	\[\forall x,y\in H_1:\quad \|\tilde{f}(x)-\tilde{f}(y)\|\leq K\|x-y\|\]
	Moreover, if $f$ is bounded (i.e. $f$ is Lip-$1$) then $\tilde{f}$ can be chosen to be bounded and such that $\sup_A \|f\|=\sup_{H_1} \|\tilde{f}\|$.
	\end{theo}
	
	 C. Fefferman has been working on a variety of versions of this problem, including one that deals with appropriately approximating $f$ with a \emph{smooth} map, which can be of enormous use in practice when one is collecting a finite sample of data (see for example \cite{Fefferman1, Fefferman2, Fefferman3}).
%%%%
%%Lipschitz vector fields and their flows
%%%%
\section{Flows of Lipschitz vector fields}
	In this section, we aim to study the Lipschitz regularity (in time and space) of flows of Lipschitz vector fields. The answer to this question in the context of smooth vector fields both in the Euclidean and manifold setting is widely covered in the literature (cf. \cite{Lee} for example). Our further aim is to quantify these results. We start by giving a (simplified) definition of the flow of a vector field then stating the fundamental theorem of ordinary differential equations (O.D.E.s) which ensures the existence and uniqueness of flows:
	
	%%%% Definition of flows %%%%
	\begin{Def} Let $I$ be an open interval. Let $M$ be a $\mathcal{C}^1$-manifold and $A$ be a vector field on $M$. A $\mathcal{C}^1$-path $\gamma:I\to M$ is said to be an integral curve of $A$ if:
	\[\forall t\in I:\quad \gamma'(t)=A(\gamma(t))\]
	If $0\in I$, we say that $\gamma(0)$ is the starting point of $\gamma$. If furthermore  $U$ denotes a subset of $M$ and $\tilde{A}:I\times U\to M$ is such that, for every $x\in U$, $t\mapsto \tilde{A}(t,x)$ is an integral curve of $A$ starting at $x$, we say then that $\tilde{A}$ is a local flow (or global flow if $I\times U=\mathbb{R}\times M$) of $A$ on $I\times U$.
	\end{Def}

	\begin{Notation} Under the assumption of existence and if there is no risk of confusion, we will be denoting by $\tilde{A}$ the flow of a vector field $A$; by $\tilde{A}_t$, for $t\in\mathbb{R}$, the map $x\mapsto \tilde{A}(t,x)$; and by $\tilde{A}_x$, for $x\in M$, the map $t\mapsto \tilde{A}(t,x)$.\end{Notation}
	
	The following fundamental theorem is classical (see for example \cite{Lee}) and can be seen as a special case of Picard-Lindel\"olf's theorem dealing with differential equations driven by paths of bounded variation (see for example \cite{Flour}).

	%%%% ODE theorem %%%%
	\begin{theo}\label{ODEExist} Let $A$ be a $\textrm{Lip}-1$ vector field on a Banach space $E$. Then there exists a unique global flow of $A$ on $\mathbb{R}\times E$.
	\end{theo}
	The following Gronwall-type comparison lemma is going to be of use to us to obtain quantitative Lipschitz bounds for flows of Lipschitz vector fields (see \cite{Gronwall} or \cite{Lee} for a version with the language used here):
	%%%% Comparison lemma %%%%
	\begin{lemma}[Comparison lemma]\label{ComparLemma} Let $I$ be an open interval and $E$ be a real inner product space. Let $u:I\to E$ be a differentiable map such that there exists $a>0$ and $b\geq 0$ such that:
	\[\forall t\in I: \|u'(t)\|\leq a\|u(t)\|+b\]
	Then, if $t_0\in I$, we have:
	\[\forall t\in I: \|u(t)\|\leq e^{a|t-t_0|}\|u(t_0)\|+\frac{b}{a}(e^{a|t-t_0|}-1)\]
	\end{lemma}
	We start by studying the regularity of flows of Lipschitz-1 vector fields. For smoother vector fields, we will naturally encounter flows of vector fields that are only locally Lipschitz. This is the reason for which we make the following general claim:
%%%% LOCAL: The Flow is Holder and bounded %%%%
	\begin{lemma}\label{FlowHolderBoundLOCAL}
	 Let $H$ be a Lip-1 vector field on a subset $V$ of a Banach space $E$. Let $T>0$ and $U$ be a subset of $E$. Let $G$ be a flow of $H$ on $(-T,T)\times U$ (we assume that $G((-T,T)\times U)\subseteq V$). Then $G$ is $1$-H\"older. More precisely, for all $t,\tilde{t}\in (-T,T)$ and $y,\tilde{y}\in U$, we have:
	\[\|G(t,y)-G(\tilde{t},\tilde{y})\|\leq e^{(|t|\wedge|\tilde{t}|)\|H\|_{\textrm{Lip}-1}}\|y-\tilde{y}\|+\|H\|_{\textrm{Lip}-1}|t-\tilde{t}|\]
	\end{lemma}
	\begin{proof}
	We start by proving that $G$ is uniformly space 1-H\"older:
	\[\forall t\in (-T,T),\forall y,\tilde{y}\in U: \|G(t,y)-G(t,\tilde{y})\|\leq e^{|t|\|H\|_{\textrm{Lip}-1}}\|y-\tilde{y}\|\]
	This is trivial in the case when $H$ is constant (on $V$). Assume then that $H$ is not constant and let $y,\tilde{y}\in U$. Define $u$ on $(-T,T)$ by the identity: $u(t)=G(t,y)-G(t,\tilde{y})$. Note that $u(0)=y-\tilde{y}$. $u$ is differentiable and:
	\[\forall t \in (-T,T): u'(t)=H(G(t,y))-H(G(t,\tilde{y}))\]
	Therefore, for all $t\in (-T,T)$, $\|u'(t)\|\leq\|H\|_{\textrm{Lip}-1}\|u(t)\|$. The comparison lemma \ref{ComparLemma} gives then the sought bound.\\
	Let $t,\tilde{t}\in (-T,T)$ and assume that $|t|\leq |\tilde{t}|$:
	\[\begin{array}{lcl}
	\|G(t,y)-G(\tilde{t},\tilde{y})\|&\leq&\|G(t,y)-G(t,\tilde{y})\|+\|G(t,\tilde{y})-G(\tilde{t},\tilde{y})\|\\
										&\leq&e^{|t|\|H\|_{\textrm{Lip}-1}}\|y-\tilde{y}\| +\|\int_{\tilde{t}}^t H(G(u,\tilde{y}))\mathrm{d}u\|\\
										&\leq&e^{|t|\|H\|_{\textrm{Lip}-1}}\|y-\tilde{y}\|+\|H\|_{\textrm{Lip}-1}|t-\tilde{t}|\\
	\end{array}\]
	Therefore, $G$ is $1$-H\"older continuous.\end{proof}
	Lemma \ref{FlowHolderBoundLOCAL} naturally gives us the following result about flows of Lipschitz-1 vector fields:
	%%%% GLOBAL:The Flow is Holder and bounded %%%%
	\begin{Cor}\label{FlowHolderBound} Let $A$ be a $\textrm{Lip}-1$ vector field on a Banach space $E$ and $\tilde{A}$ its global flow. Then:
	\begin{itemize}
	\item $\forall t\in \mathbb{R},\forall y,\tilde{y}\in E: \|\tilde{A}(t,y)-\tilde{A}(t,\tilde{y})\|\leq e^{|t|\|A\|_{\textrm{Lip}-1}}\|y-\tilde{y}\|$.
	\item $\tilde{A}$ is locally $1$-H\"older: for all $t,\tilde{t}\in \mathbb{R}$ and $y,\tilde{y}\in E$:
	\[\|\tilde{A}(t,y)-\tilde{A}(\tilde{t},\tilde{y})\|\leq e^{(|t|\wedge|\tilde{t}|)\|A\|_{\textrm{Lip}-1}}\|y-\tilde{y}\|+\|A\|_{\infty}|t-\tilde{t}|\]
	\item $\forall T,r\in\mathbb{R}^*_+, \forall x_0\in E:\quad\tilde{A}((-T,T)\times B(x_0,r))\subseteq B(x_0,r+T\|A\|_{\textrm{Lip}-1})$. 
	\end{itemize}
	\end{Cor}
	\begin{proof}
	A straightforward consequence of lemma \ref{FlowHolderBoundLOCAL}.\end{proof}	
	We show now that the flows of differentiable vector fields are differentiable too:
%%%% LOCAL The Flow is differentiable in space-time %%%%
	\begin{lemma}\label{LocalFlowisDiff} Let $0<\varepsilon \leq 1$ and $d\in \mathbb{N}^*$. Let $H$ be a Lip-$(1+\varepsilon)$ vector field defined on a subset $V$ of $\mathbb{R}^d$. Let $T>0$ and $U$ be an open subset of $\mathbb{R}^d$. Let $G$ be a flow of $H$ on $(-T,T)\times U$ (we assume that $G((-T,T)\times U)\subseteq V$). Then $G$ is continuously differentiable on $(-T,T)\times U$ and, if $(\vec{e}_1,\ldots,\vec{e}_d)$ is a basis for $\mathbb{R}^d$, then for all $(t,y)\in (-T,T)\times U$:
	\[\|\partial_t{G}(t,y)\|\leq\|H\|_{\textrm{Lip}-(1+\varepsilon)}\]
	and
	\[\|\partial_{x_i}{G}(t,y)\|\leq e^{T\|H\|_{\textrm{Lip}-1}}\|\vec{e}_i\|\]
	\end{lemma}
	\begin{proof} The result being trivial for $H=0$ (on $V$), we assume that $H\neq 0$. Let $y\in U$. By definition of the flow:
	\[ \forall t\in (-T,T):\quad G(t,y)=y+\int_0^t H(G(u,y))\mathrm{d}u\]
	As both $H$ and $G$ are continuous, $G(.,y)$ is then continuously differentiable and, for all $t\in\mathbb (-T,T): \partial_t{G}(t,y)=H(G(t,y))$. Moreover:
	\[\forall (t,y)\in(-T,T)\times U: \|\partial_t{G}(t,y)\|\leq\|H\|_{\infty,V} \leq\|H\|_{\textrm{Lip}-(1+\varepsilon)}\]
	We will prove now that $G$ is continously differentiable in space. Let $(\vec{e}_1,\ldots,\vec{e}_d)$ be a basis for $\mathbb{R}^d$. Let $i\in [\![1,d]\!]$. For $h\in \mathbb{R}^*$, we define the map $\Delta^i_h$ on $(-T,T)\times U$ by the relation:
	\[\Delta^i_h(t,y)=\frac{G(t,y+h\vec{e}_i)-G(t,y)}{h}\]
	We are going to show that the sequence $(\Delta^i_h)_{|h|>0}$ converges uniformly on $(-T,T)\times U$ (as $h$ goes to zero). Let $h\in \mathbb{R}^*$. For $(t,y)\in(-T,T)\times U$, lemma \ref{FlowHolderBoundLOCAL} gives the inequality:
	\begin{equation}\label{Controlderive}\|\Delta^i_h(t,y)\|\leq e^{T\|H\|_{\textrm{Lip}-1}}\|\vec{e}_i\| 
	\end{equation}
	$H$ being $\textrm{Lip}-(1+\varepsilon)$, let $S$ be the remainder map defined on $V^2$ with values in $\mathbb{R}^d$ such that, for all $a,b \in V$:
	\[ H(a)=H(b) + \mathrm{d}H(b)(a-b) + S(a,b)\]
	and\[ \|S(a,b)\|\leq\|H\|_{\textrm{Lip}-(1+\varepsilon)}\|a-b\|^{1+\varepsilon}\]
	$\Delta^i_h$ is obviously continuously differentiable in time. Let $(t,y)\in(-T,T)\times U$:
	\[\partial_t{\Delta^i_h}(t,y)=\mathrm{d}H(G(t,y))(\Delta^i_h(t,y)) + \frac{1}{h}S(G(t,y+h\vec{e}_i),G(t,y))\]
	Let $\tilde{h}\in \mathbb{R}^*$. From the calculation above and the inequality (\ref{Controlderive}), we get that:
	\[\begin{array}{lcl}
	\|\partial_t{\Delta^i_h}(t,y)-\partial_t{\Delta^i_{\tilde{h}}}(t,y)\|&\leq&\|\mathrm{d}H(G(t,y))(\Delta^i_h(t,y)-\Delta^i_{\tilde{h}}(t,y))\|+\\
															&&\|\frac{1}{h}S(G(t,y+h\vec{e}_i),G(t,y))\|+ \|\frac{1}{\tilde{h}}S(G(t,y+\tilde{h}\vec{e}_i),G(t,y))\|\\
															&\leq&\|H\|_{\textrm{Lip}-(1+\varepsilon)}(\|\Delta^i_h(t,y)-\Delta^i_{\tilde{h}}(t,y)\|+ \\
															&&|h|^{\varepsilon}\|\Delta^i_h(t,y)\|^{1+\varepsilon}+|\tilde{h}|^{\varepsilon}\|\Delta^i_{\tilde{h}}(t,y)\|^{1+\varepsilon})\\
															&\leq&\|H\|_{\textrm{Lip}-(1+\varepsilon)}(\|\Delta^i_h(t,y)-\Delta^i_{\tilde{h}}(t,y)\|+\\
															&&2(|h|\vee|\tilde{h}|)^{\varepsilon}(e^{T\|H\|_{\textrm{Lip}-1}}\|\vec{e}_i\|)^{1+\varepsilon})\\
	\end{array}\]
	Therefore, using the comparison lemma and the fact that $\Delta^i_h(0,y)=\Delta^i_{\tilde{h}}(0,y)=\vec{e}_i$, we get the following inequality:
	\[\|\Delta^i_h-\Delta^i_{\tilde{h}}\|_{\infty,(-T,T)\times U}\leq 2(|h|\vee|\tilde{h}|)^{\varepsilon}(e^{T\|H\|_{\textrm{Lip}-1}}\|\vec{e}_i\|)^{1+\varepsilon}(e^{T\|H\|_{\textrm{Lip}-(1+\varepsilon)}}-1)\]
	We therefore see that $(\Delta^i_h)_{|h|>0}$ converges uniformly on $(-T,T)\times U$ and that $\partial_{x_i}{G}$ exists (as its limit). As, for every $h$, $\Delta^i_h$ is continuous then $\partial_{x_i}{G}$ is also continuous. We also get the following bound by passing to the limit in the inequality (\ref{Controlderive}):
	\[\forall (t,y) \in (-T,T)\times U: \quad \|\partial_{x_i}{G}(t,y)\|\leq e^{T\|H\|_{\textrm{Lip}-1}}\|\vec{e}_i\|\]\end{proof}
		%%%% The Flow is differentiable in space-time %%%%
	\begin{Cor}\label{GlobalFlowisDiff} Let $0<\varepsilon \leq 1$ and $d\in \mathbb{N}^*$. Let $A$ be a $\textrm{Lip}-(1+\varepsilon)$ vector field on $\mathbb{R}^d$ and $\tilde{A}$ be its global flow. Then $\tilde{A}$ is continuously differentiable and, if $(\vec{e}_1,\ldots,\vec{e}_d)$ is a basis for $\mathbb{R}^d$, then for all $(t,y)\in\mathbb{R}\times \mathbb{R}^d$:
	\[\|\partial_t{\tilde{A}}(t,y)\|\leq\|A\|_{\textrm{Lip}-(1+\varepsilon)}\]
	and
	\[\|\partial_{x_i}{\tilde{A}}(t,y)\|\leq e^{|t|\|A\|_{\textrm{Lip}-1}}\|\vec{e}_i\|\]
	\end{Cor}
	\begin{proof}
	This is a version of the previous lemma \ref{LocalFlowisDiff} where $U=V=\mathbb{R}^d$.
	\end{proof}
	\begin{Rem} It is worth noting that without any additional assumption on the vector field $A$ (other than it being $\textrm{Lip}-1$), one can show that, for every $y \in \mathbb{R}^d$, $\tilde{A}_y$ is actually $\textrm{Lip}-2$ on every bounded interval of time. However, for the techniques above to work, we need $A$ to be a little  smoother than $\textrm{Lip}-1$ ( $\textrm{Lip}-(1+\varepsilon)$ as in the above lemma, but $\mathcal{C}^1$ is enough) to have the space-differentiability of $\tilde{A}$.
	\end{Rem}
%%%% The Local Flow is almost Lipschitz %%%%
	\begin{lemma}\label{LocalFlowisLip} Let $0<\varepsilon \leq 1$ and $d\in \mathbb{N}^*$. Let $H$ be a Lip-$(1+\varepsilon)$ vector field defined on a subset $V$ of $\mathbb{R}^d$. Let $T>0$ and $U$ be an open subset of $\mathbb{R}^d$. Let $G$ be a flow of $H$ on $(-T,T)\times U$ (we assume that $G((-T,T)\times U)\subseteq V$). Then $\mathrm{d}G$ is Lipschitz-$\varepsilon$ on $(-T,T)\times U$ and its Lip-$\varepsilon$ norm is bounded from above by a constant depending only on $\|H\|_{\textrm{Lip}-(1+\varepsilon)}$, $\varepsilon$ and $T$.
	\end{lemma}
	\begin{proof}
	The lemma is trivial in the case where $H=0$. Assume then that $H\neq 0$. In the following, we endow $\mathbb{R}\times \mathbb{R}^d$ with the $l^{\infty}$ norm that we will denote by $N$.\\
	Let $(s,z) \in (-T,T)\times U$. Let $(t,y), (\tilde{t},\tilde{y})\in B((s,z),1/2)\cap((-T,T)\times U)$. Using the definition of a $\textrm{Lip}-1$ map and our preliminary study of $G$ (lemma \ref{FlowHolderBoundLOCAL}), we get the following inequality:
	\[\begin{array}{lcl}
	\|\partial_t{G}(t,y)-\partial_t{G}(\tilde{t},\tilde{y})\|&\leq&\|H\|_{\textrm{Lip}-1}\|G(t,y)-G(\tilde{t},\tilde{y})\|\\
															&\leq&\|H\|_{\textrm{Lip}-1}(e^{T\|H\|_{\textrm{Lip}-1}}\|y-\tilde{y}\|+\|H\|_{\textrm{Lip}-1}|t-\tilde{t}|)\\
															&\leq&\|H\|_{\textrm{Lip}-1}(e^{T\|H\|_{\textrm{Lip}-1}}+\|H\|_{\textrm{Lip}-1})N((t,y)-(\tilde{t},\tilde{y}))^{\varepsilon}\\
	\end{array}\]
	Hence, $\partial_t{G}$ is $\varepsilon$-H\"older on $B((s,z),1/2)\cap((-T,T)\times U)$.\\
	Let $(\vec{e}_1,\ldots,\vec{e}_d)$ be a basis for $\mathbb{R}^d$. Let $i\in[\![1,d]\!]$ and $(y,\tilde{y})\in U^2$. Define the map $v$ on $(-T,T)$ by the identity $v(t)=\partial_{x_i}{G}(t,y)-\partial_{x_i}{G}(t,\tilde{y})$. Since  $\partial_{x_i}{G}$ satisfies the following differential equation:
	\[\forall t \in (-T,T): \quad\partial_{x_i}{G}(t,y)=\vec{e}_i+\int_0^t\mathrm{d}H(G(u,y))\partial_{x_i}{G}(u,y)\mathrm{d}u\]
	$v$ is then continuously differentiable. Using the fact that $H$ is $\textrm{Lip}-(1+\varepsilon)$ on $V$, the controls obtained in lemmas \ref{FlowHolderBoundLOCAL} and \ref{LocalFlowisDiff}, we get the following inequality, for all $t\in(-T,T)$:
	\[\|v'(t)\|\leq \|\mathrm{d}H\|_{\textrm{Lip}-\varepsilon} e^{(1+\varepsilon) T\|H\|_{\textrm{Lip}-1}}\|\vec{e}_i\|\|y-\tilde{y}\|^{\varepsilon} + \|\mathrm{d}H\|_{\textrm{Lip}-\varepsilon}\|v(t)\|\]
	Noting that $\|H\|_{\textrm{Lip}-(1+\varepsilon)}>0$ and using the comparison lemma, we then get, for all $t \in (-T,T)$:
	\[ \|\partial_{x_i}{G}(t,y)-\partial_{x_i}{G}(t,\tilde{y})\|\leq e^{(1+\varepsilon) T\|H\|_{\textrm{Lip}-1}}\|\vec{e}_i\|\|y-\tilde{y}\|^{\varepsilon} (e^{T\|H\|_{\textrm{Lip}-(1+\varepsilon)}}-1)\]
	Let $(t,\tilde{t})\in (-T,T)^2$. Using successively the differential equation satisfied by $\partial_{x_i}{G}(.,\tilde{y})$, the fact that $H$ is $\textrm{Lip}-(1+\varepsilon)$ and finally the lemma \ref{LocalFlowisDiff}, one gets: 
	\[\begin{array}{lcl}
	\|\partial_{x_i}{G}(t,\tilde{y})-\partial_{x_i}{G}(\tilde{t},\tilde{y})\|\|&=&\|\int_{\tilde{t}}^t(\mathrm{d}H(G(u,\tilde{y}))\partial_{x_i}{G}(u,\tilde{y})\mathrm{d}u\|\\
																&\leq&|t-\tilde{t}|\|\mathrm{d}H\|_{\infty}\|\partial_{x_i}{G}(.,\tilde{y})\|_{\infty,[\tilde{t},t]}\\
																&\leq&|t-\tilde{t}|\|H\|_{\textrm{Lip}-(1+\varepsilon)} e^{T\|H\|_{\textrm{Lip}-1}}\|\vec{e}_i\|\\
	\end{array}\]
	Finally we get from all the above the following inequality for every $(t,y), (\tilde{t},\tilde{y})\in B((s,z),1/2)\cap((-T,T)\times U)$:
	\[\begin{array}{lcl}
	\|\partial_{x_i}{G}(t,y)-\partial_{x_i}{G}(\tilde{t},\tilde{y})\|&\leq&(e^{T(\varepsilon \|H\|_{\textrm{Lip}-1}+\|H\|_{\textrm{Lip}-(1+\varepsilon)})} +\|H\|_{\textrm{Lip}-(1+\varepsilon)})\\
																&&e^{T\|H\|_{\textrm{Lip}-1}}\|\vec{e}_i\| N((t,y)-(\tilde{t},\tilde{y}))^{\varepsilon}\\
	\end{array}\]
	Therefore, $\partial_{x_i}{G}$ is $\varepsilon$-H\"older on $ B((s,z),1/2)\cap((-T,T)\times U)$. Define the following constants:
	\[\left\{\begin{array}{lcl}
	M_1&=&\|H\|_{\textrm{Lip}-1 }(e^{T\|H\|_{\textrm{Lip}-1 }}+\|H\|_{\textrm{Lip}-1 })\\
	M_2&=&(e^{T(\varepsilon \|H\|_{\textrm{Lip}-1 }+\|H\|_{\textrm{Lip}-(1+\varepsilon) })} +\|H\|_{\textrm{Lip}-(1+\varepsilon) }) e^{T\|H\|_{\textrm{Lip}-1 }}\underset{1\leq i \leq d}\sum \|\vec{e}_i\|\\
	M_{\infty}&=&\|H\|_{\textrm{Lip}-(1+\varepsilon) }+ e^{T\|H\|_{\textrm{Lip}-1 }}\underset{1\leq i \leq d}\sum\|\vec{e}_i\|\\
	\end{array}\right.\]
	When restricted to $B((s,z),1/2)\cap((-T,T)\times U)$, $\mathrm{d}G$ is Lipschitz-$\varepsilon$ with norm upper-bounded by $\max(M_1+M_2,M_{\infty})$. Therefore $\mathrm{d}G$ is Lipschitz-$\varepsilon$ on $(-T,T)\times U$ and there exists a constant $c_{\varepsilon}$ such that:
	\[\|\mathrm{d}G\|_{\textrm{Lip}-\varepsilon}\leq c_{\varepsilon} \max(M_1+M_2,M_{\infty})\]	
	\end{proof}
	%%%% The Global Flow is almost Lipschitz %%%%
	\begin{Cor}\label{GlobalFlowisLip} Let $0<\varepsilon \leq 1$. Let $d\in\mathbb{N}^*$ and $T>0$. Let and $A$ be a $\textrm{Lip}-(1+\varepsilon)$ vector field on $\mathbb{R}^d$ and $\tilde{A}$ its global flow. Then $\mathrm{d}\tilde{A}$ is Lipschitz-$\varepsilon$ on $(-T,T)\times \mathbb{R}^d$ and its Lip-$\varepsilon$ norm is bounded from above by a constant depending only on $\|A\|_{\textrm{Lip}-(1+\varepsilon)}$, $\varepsilon$ and $T$.
	\end{Cor}
	\begin{proof} This is a special case of lemma \ref{LocalFlowisLip}, with $U=V=\mathbb{R}^d$.
	\end{proof}
	Before we turn our intention to the main theorem of this section, we first make the link, in terms of Lipschitz regularity, between the derivative $\mathrm{d}H$ of a smooth map $H$ and the representation $\mathrm{d}H: (x,y)\mapsto \mathrm{d}H(x)(y)$.
	%%%%
	%%%% Derivatives are semi-linear Lipschitz
	%%%%
	\begin{lemma}\label{DerivativeIsTSLipschitz} Let $\gamma> 0$, $M>0$ and $E, F$ and $G$ be three normed vector spaces. We assume that $((E\otimes F)^{\otimes k})_{k\geq 1}$ are endowed with norms satisfying the projective property. Let $V$ be a subset of $E$ and $f:V\to \mathcal{L}( F,G)$ be a Lip-$\gamma$ map. Then
	\[\begin{array}{rccl}
	\widehat{f} :& V \times F& \longrightarrow & G \\
		& (x_1,x_2) & \longmapsto & f(x_1)(x_2)\\
	\end{array}
	\]
	is Lip-$\gamma$ when restricted to $V \times B(0,M)$. Moreover, there exists a constant $m_{\gamma}$ depending only on $\gamma$ such that:
	\[\|\widehat{f}\|_{\textrm{Lip}-\gamma}\leq m_{\gamma} \max(1,M) \|f\|_{\textrm{Lip}-\gamma}\]
	If $\gamma>1$, then one can take $m_{\gamma}=1$.
	\end{lemma}
	\begin{proof}
	Let $H_1$ and $H_2$ be the maps defined on $V \times F$ by:
	\[H_1(x_1,x_2)=f(x_1)\in \mathcal{L}_c(F,G) \quad ; \quad H_2(x_1,x_2)=x_2 \in F\]
	Both maps are obviously Lip-$\gamma$ on $V \times B(0,M)$. On the one hand, we have $\|H_1\|_{\textrm{Lip}-\gamma}= \|f\|_{\textrm{Lip}-\gamma}$. On the other hand, if $\gamma>1$ then $\|H_2\|_{\textrm{Lip}-\gamma}\leq \max(1,M)$ and if $\gamma\leq 1$, then $H_2$ is Lip-$\gamma$ on the intersection of open balls of radius less than 1/2 with $V \times B(0,M)$ with norm less than $\max(1,M)$ and using lemma \ref{LocalLipGen} there exists a constant $m_{\gamma}$ depending only on $\gamma$  such that 
	\[ 	\|H_2\|_{\textrm{Lip}-\gamma}\leq m_\gamma \max(1,M) \]
	Now consider the continuous bilinear map:
	\[\begin{array}{rccl}
	B :& \mathcal{L}_c(F,G) \times F& \longrightarrow & G \\
		& (u,z) & \longmapsto & u(z)\\
	\end{array}
	\]
	Then $\widehat{f}=B(H_1,H_2)$ and is Lip-$\gamma$ by proposition \ref{Compobifunc}. Moreover:
	\[\|\widehat{f}\|_{\textrm{Lip}-\gamma}\leq m_\gamma \max(1,M) \|f\|_{\textrm{Lip}-\gamma}\]
	\end{proof}	
		Loosely speaking, the following lemma is a converse of the previous one: for a two-variable map that is linear in one of its variables (like Fr\'echet derivatives of smooth maps), it is enough to show that it is Lipschitz on a bounded set:
	%%%%
	%%%% Lip semi linear are globally Lip
	%%%%
	\begin{lemma}\label{LipSLareGLip} Let $\gamma>0$, $r>0$, $d \in \mathbb{N}^*$ and $E$ and $G$ be two normed vector spaces. Let $U$ be a subset of $\mathbb{R}^d$.  We assume that there exists $C\in\mathbb{R}$ such that for all $k\in[\![1,n]\!]$, the norms on $(E\times \mathbb{R}^d)^{\otimes k}$ and $E^{\otimes k}$ satisfy the following inequality:
	\[\forall v_1,\ldots, v_k \in E: \Big\| (v_1,0)\otimes \cdots \otimes (v_k,0) \Big\|
	\leq C \Big\| v_1\otimes \cdots \otimes v_k\Big\|\]
	Let $f:U\times \mathbb{R}^d\to E$ be a map linear with respect to its second variable such that $f_0:=f_{|U\times B(0,r)}$ is Lip-$\gamma$. Then:
	\[\begin{array}{cccl}
	\widehat{f}:&	U & \longrightarrow & \mathcal{L}(\mathbb{R}^d,G)\\
			& x	& \longmapsto &  f(x,.)\\
	\end{array}
	\]
	is Lip-$\gamma$ and we have $\|\widehat{f}\|_{\textrm{Lip}-\gamma}\leq \max(1,C)\|f\|_{\textrm{Lip}-\gamma}$.
	\end{lemma}
	\begin{proof} We denote by $(e_1,\ldots,e_d)$ a basis of $\mathbb{R}^d$ and by $(e_1^{*},\ldots,e_d^{*})$ its dual basis. Without loss of generality, we assume that for all $i\in [\![1,d]\!]$. $\|e_i\|<r$. Now consider the maps:
	\[\begin{array}{ccclccccl}
	u:&	G^d & \longrightarrow & \mathcal{L}((\mathbb{R}^d),G) &\quad \textrm{and}\quad & Z:&	U & \longrightarrow &G^d\\
		& (v_i)_{1\leq i \leq d}	& \longmapsto & \sum v_ie_i^*&	&& x	& \longmapsto & (f(x,e_i))_{1\leq i \leq d} \\
	\end{array}
	\]
	then $\widehat{f}=u \circ Z$. It is an easy exercise to show that $Z$ is Lip-$\gamma$ and that
	\[\|Z\|_{\textrm{Lip}-\gamma}\leq \max(1,C)\|f\|_{\textrm{Lip}-\gamma}\] 
	As $u$ is linear with norm 1 then, by proposition \ref{Compolinfunc}, $\widehat{f}$ is Lip-$\gamma$ with the required upper-bound on its Lipschitz norm.
	\end{proof}
	Finally, we show that flows of Lipschitz vector fields are also Lipschitz on bounded sets and have a well-controlled Lipschitz norm:
	%%%% FLOW LOCAL SMOOTHNESS THEOREM %%%%
	\begin{theo}\label{FlowLocSmoothTheoGen} Let $0<\varepsilon \leq 1$ and $n,d\in \mathbb{N}^*$. Let $H$ be a Lip-$(n+\varepsilon)$ vector field defined on a subset $V$ of $\mathbb{R}^d$. Let $T>0$ and $U$ be an open convex subset of $\mathbb{R}^d$. Let $G$ be a flow of $H$ on $(-T,T)\times U$ (assuming that $G((-T,T)\times U)\subseteq V$). Then $\mathrm{d}G$ is Lipschitz-$(n+\varepsilon-1)$ on $(-T,T)\times U$ and its Lip-$(n+\varepsilon-1)$ norm is bounded from above by a constant depending only on $\|H\|_{\textrm{Lip}-(n+\varepsilon)}$, $\varepsilon$ and $T$.
	\end{theo}
	\begin{proof}
	% Case 1 + epsilon
	We will prove the theorem by induction. Lemma \ref{LocalFlowisLip} deals with the case $n=1$.
	%Induction
	%% Induction hypothesis
	Let $n\in\mathbb{N}^*$. Assume that the assertion is true for $\textrm{Lip}-(n+\varepsilon)$ vector fields and let us prove it when $H$ is $\textrm{Lip}-(n+1+\varepsilon)$. By the induction hypothesis, we know that $\mathrm{d}G$ is Lipschitz-$(n+\varepsilon-1)$ on $(-T,T)\times U$ and there exists a constant $C$ depending only on $T$, $n$, $\varepsilon$ and $\|H\|_{\textrm{Lip}-(n+\varepsilon)}$ such that $\|\mathrm{d}G\|_{\textrm{Lip}-(n+\varepsilon-1)}\leq C$. In particular, by proposition \ref{AlmostLipDiffMaps}, $G$ is almost Lip-$(n+\varepsilon)$.
	%% Smoothmess in time
	$G$ being a flow of $H$, then $\partial_{t}{G}=H\circ G$. As $H$ is Lip-$(n+\varepsilon)$ and $G$ is almost Lip-$(n+\varepsilon)$, then, by theorem \ref{CompolipfuncCASE22}, $\partial_{t}{G}$ is $\textrm{Lip}-(n+\varepsilon)$ and $\|\partial_{t}{G}\|_{\textrm{Lip}-(n+\varepsilon)}$ can be upper-bounded by a constant depending only on $T$, $n$, $\varepsilon$ and $\|H\|_{\textrm{Lip}-(n+1+\varepsilon)}$.\\
	%% Writing the space derivative as a flow
	Let us denote by $\mathrm{d}_x G$ the spatial derivative of $G$, i.e.:
	\[\begin{array}{rccl}
	\mathrm{d}_x G: & (-T,T)\times U & \longrightarrow &\mathcal{L}(\mathbb{R}^d,\mathbb{R}^d) \\
		& (t,y) & \longmapsto & (b\mapsto \mathrm{d}G(t,y)(0,b)=\sum_{1}^d b_i\partial_{x_i}G(t,y))\\
	\end{array}
	\]
	Let $\tilde{A}$ be defined as:
	\[\begin{array}{rccl}
	\tilde{A}: & (-T,T)\times U \times\mathbb{R}^d & \longrightarrow & V\times\mathbb{R}^d \\
		& (t,y,b) & \longmapsto & (G(t,y), \mathrm{d}_x G(t,y)(b))\\
	\end{array}
	\]
	Then $\tilde{A}(0,y,b)=(y, b)$. Moreover, $\tilde{A}$ is differentiable in the time variable and:
	\[\partial_t \tilde{A}(t,y,b)=(H(G(t,y)), \mathrm{d}H(G(t,y))(\mathrm{d}_x G(t,y)(b)))\]
	Now define the vector field:
	\[\begin{array}{rccl}
	A: & V\times\mathbb{R}^d& \longrightarrow & \mathbb{R}^d\times\mathbb{R}^d \\
		& (x_1,x_2) & \longmapsto & (H(x_1), \mathrm{d}H(x_1)(x_2))\\
	\end{array}
	\]
	so that $\partial_t \tilde{A}(t,y,b)=A(\tilde{A}(t,y,b))$ and $\tilde{A}$ is the flow of $A$ on $(-T,T)\times U \times\mathbb{R}^d$. Moreover, by our assumption on the values of $G$ and lemma \ref{LocalFlowisDiff}:
	\[\forall r>0:\quad \tilde{A}((-T,T)\times U\times B(0,r))\subseteq
	 V \times  
	B(0,r e^{T\|H\|_{\textrm{Lip}-1}})
	\]
	%% Smoothness in space
	By lemma \ref{DerivativeIsTSLipschitz} , $A$ is $\textrm{Lip}-(n+\varepsilon)$ on $V\times B(0,1)$ and there exists a constant $M_\gamma$ depending only on $\gamma$ such that:
	\[\|A\|_{\textrm{Lip}-(n+\varepsilon), V\times B(0,1)}\leq m_\gamma \|H\|_{\textrm{Lip}-(n+\varepsilon)} \]
	By the induction hypothesis, we can claim that $\tilde{A}$ is almost $\textrm{Lip}-(n+\varepsilon)$ on $(-T,T)\times U \times B(0,e^{-T\|H\|_{\textrm{Lip}-1}})$ and so is the map $(t,y,b)\rightarrow\mathrm{d}_x G(t,y)(b)$. Since the latter is also bounded (lemma \ref{LocalFlowisDiff}), then it is $\textrm{Lip}-(n+\varepsilon)$ on $(-T,T)\times U \times B(0,e^{-T\|H\|_{\textrm{Lip}-1}})$. Hence (lemma \ref{LipSLareGLip}), $\mathrm{d}_x G$ is $\textrm{Lip}-(n+\varepsilon)$ on $(-T,T)\times U$ which gives the result.
	\end{proof}
	%%%% Global FLOW SMOOTHNESS THEOREM %%%%
	\begin{Cor}\label{FlowGlobalSmoothTheoGen} Let $n,d \in \mathbb{N}^*$ and $0<\varepsilon \leq 1$. Let $A$ be a $\textrm{Lip}-(n+\varepsilon)$ vector field on $\mathbb{R}^d$. Then $\tilde{A}$ is almost $\textrm{Lip}-(n+\varepsilon)$ on $(-T,T)\times \mathbb{R}^d$ and there exists a constant $C$ depending only on $T$, $n$, $\varepsilon$ and $\|A\|_{\textrm{Lip}-(n+\varepsilon)}$ such that $\|\mathrm{d}\tilde{A}\|_{\textrm{Lip}-(n+\varepsilon)}\leq C$.
	\end{Cor}
	
%%%%
%%Constant Rank/Inverse function theorems for Lipschitz maps
%%%%
	\section{Constant Rank theorems for Lipschitz maps}
	The two versions of the constant rank theorem in this section and the related techniques are classical in the case of smooth maps and the literature is abundant in this matter (see for example \cite{Lee}). As the reader may notice, and in the same spirit of almost Lipschitz maps, we will only be assuming that the derivatives are Lipschitz (instead of the maps themselves) as this is a less demanding requirement to get our quantitative estimates.\\
	
	We will be working in the finite-dimensional case and will assume that the norms on tensor spaces satisfy all the norm properties presented in section \ref{sec:lip-maps}. Finite dimensional vector spaces are endowed with the $l^{\infty}$ norm while norms of continuous linear maps and matrices are computed as subordinate norms.	The statement of the theorems and their subsequent proofs adapt easily in the case of other norms.
	
	%%
	%Inverse function theorem
	%%
	\subsection{The inverse function theorem}
	
	When working with a Lipschitz map that is of maximal rank at a given point, one can quantify the size of the domain on which said map stays of maximal rank:
	%%%%
	%%%% Quantification of INVERTIBILITY OF FUNCTIONS 
	%%%%
	\begin{lemma}\label{QuantInvertibility} Let $\gamma,M_1$ and $M_2$ be three positive real numbers. Let $E$ and $H$ be normed vector spaces. Let $U$ be a subset of $E$, $x_0\in U$ and $f:U\to H$ be a Lip-$\gamma$ map such that $\|f\|_{\textrm{Lip-}\gamma}\leq M_1$. Then:
	\begin{enumerate}
	\item There exists $\delta >0$ depending only on $\gamma$ and $M_1M_2$ such that:
	\[\forall x \in \overline{B(x_0,\delta)}\cap U: \quad \|f(x)-f(x_0)\|\leq \frac{1}{2 M_2}\]
	 In particular, if $H$ is a Banach algebra, $f(x_0)$ is invertible and $\|f(x_0)^{-1}\|\leq M_2$, then $f$ is invertible on $\overline{B(x_0,\delta)}\cap U$.
	\item Consider the case of $H=\mathcal{M}_{m,p}(\mathbb{R})$ endowed with an algebra norm $\|.\|$, with $m,p \in \mathbb{N}^*$. Assume that $f=(f_{i,j})_{(i, j)\in[\![1,m]\!]\times [\![1,p]\!]}$ is of rank less than or equal to $k\in\mathbb{N}$ and that $f(x_0)$ is of maximal rank $k$. Let $(i_1,\ldots,i_k)$ and $(j_1,\ldots,j_k)$ be, respectively, strictly ordered subsets of $[\![1,m]\!]$ and $[\![1,p]\!]$ such that $M=(f_{i_r,j_l}(x_0))_{1\leq r,l\leq k}$ is invertible. We assume that $\|M^{-1}\|\leq M_2$. Then there exists $\delta >0$ that depends only on $\gamma$ and $M_1M_2$ such that, for all $x \in \overline{B(x_0,\delta)}\cap U$, $f(x)$ is of rank $k$.
	\end{enumerate}
	\end{lemma}
	\begin{proof} \begin{enumerate}\item Let $n\in \mathbb{N}$ and $\varepsilon\in (0,1]$ such that $\gamma=n+\varepsilon$. Then, by the Taylor expansion of $f$ around $x_0$ as a Lipschitz map, we get for all $x \in U$:
	\[ \|f(x)-f(x_0)\|\leq M_1\left(\sum\limits_{k=1}^{n} \frac{\|x-x_0\|^k}{k!} + \|x-x_0\|^{n+\varepsilon}\right)\]
	It suffices to choose $\delta$ such that:
	\[\forall 0\leq t\leq\delta: \quad \sum\limits_{k=1}^{n} \frac{t^k}{k!} + t^{n+\varepsilon}\leq \frac{1}{2M_1M_2}\]
	which proves the claim. Assume now that $H$ is a Banach algebra, $f(x_0)$ is invertible and $\|f(x_0)^{-1}\|\leq M_2$. Then all elements of $B(f(x_0),\|f(x_0)^{-1}\|^{-1})$ are invertible. Since $M_2^{-1}\leq \|f(x_0)^{-1}\|^{-1} $, we have the result.
	\item The previous result insures that we can find $\delta>0$ that depends only on $\gamma$ and $M_1M_2$ such that the square matrix $(f_{i_r,j_l}(x))_{1\leq r,l\leq k}$ is invertible for all $x \in \overline{B(x_0,\delta)}\cap U$. Therefore the rank of $(f_{i,j}(x))_{(i, j)\in[\![1,m]\!]\times [\![1,p]\!]}$ is larger than or equal to $k$ on  $\overline{B(x_0,\delta)}\cap U$. Since the rank of $f$ is always less than or equal to $k$, then $f(x)$ is necessarily of rank $k$ on $\overline{B(x_0,\delta)}\cap U$.
	\end{enumerate}\end{proof}
	\begin{Rem} It is clear that the results of lemma \ref{QuantInvertibility} remain essentially true in the case of almost Lipschitz maps.
	\end{Rem}
	
	%%%% Injectivity inequalities %%%%
	\begin{lemma}\label{InjInequalities} Let $\gamma>0$. Let $U$ be an open convex subset of a Banach space $E$. let $\varphi: U\to E$ be a differentiable map such that $\mathrm{d}\varphi$ is $\textrm{Lip}-\gamma$. Let $x_0\in E$ and assume that $\mathrm{d}\varphi(x_0)$ is invertible. Then, for every $M_1>0$ and $M_2>0$ such that $\|\mathrm{d}\varphi\|_{\textrm{Lip}-\gamma}\leq M_1$ and $\|\mathrm{d}\varphi(x_0)^{-1}\|\leq M_2$, there exists a positive constant $\delta$, depending only on $\gamma$ and $M_1 M_2$, such that we have the following inequalities for all $x,\tilde{x} \in \overline{B(x_0,\delta)}\cap U$:
	\[ \|\mathrm{d}\varphi(x)-\mathrm{d}\varphi(x_0)\|\leq \frac{1}{2 M_2}\]
	\[ \|\mathrm{d}\varphi(x_0)^{-1}(\varphi(x)-\varphi(\tilde{x}))-(x-\tilde{x})\|\leq \frac{1}{2}\|x-\tilde{x}\|\]
	\[ \frac{1}{2}\|x-\tilde{x}\|\leq\|\mathrm{d}\varphi(x_0)^{-1}(\varphi(x)-\varphi(\tilde{x}))\|\leq\frac{3}{2}\|x-\tilde{x}\|\]
	In particular, $\varphi$ is injective on $\overline{B(x_0,\delta)}\cap U$.
	\end{lemma}
	\begin{proof} Let $M_1, M_2>0$ and assume that $\|\mathrm{d}\varphi\|_{\textrm{Lip}-\gamma}\leq M_1$ and $\|\mathrm{d}\varphi(x_0)^{-1}\|\leq M_2$.
	%\varphi is injective
	As $\mathrm{d}\varphi$ is Lipschitz, then by lemma \ref{QuantInvertibility}, we can find $\delta >0$ depending only on $\gamma$ and $M_1M_2$ such that:
	\[\forall x \in \overline{B(x_0,\delta)}\cap U: \quad \|\mathrm{d}\varphi(x)-\mathrm{d}\varphi(x_0)\|\leq \frac{1}{2 M_2}\left(\leq \frac{1}{2\|\mathrm{d}\varphi(x_0)^{-1}\|}\right)\]
	For all $x,\tilde{x} \in \overline{B(x_0,\delta)}\cap U$, we have then:
	\[\|\mathrm{d}\varphi(x_0)^{-1}\circ\mathrm{d}\varphi(x)-\mathrm{Id}\|\leq \frac{1}{2}\]
	and consequently (since $\overline{B(x_0,\delta)}\cap U$ is convex):
	\[\|\mathrm{d}\varphi(x_0)^{-1}(\varphi(x)-\varphi(\tilde{x}))-(x-\tilde{x})\|\leq \frac{1}{2}\|x-\tilde{x}\|\]
	Hence:
	\[\frac{1}{2}\|x-\tilde{x}\|\leq\|\mathrm{d}\varphi(x_0)^{-1}(\varphi(x)-\varphi(\tilde{x}))\|\leq\frac{3}{2}\|x-\tilde{x}\|\]			 	\end{proof}	
	%%%% DEFINITION LIP DIFFEOMORPHISM %%%%
	\begin{Def} Let $\gamma > 0$. Let $E$ and $F$ be two normed vector spaces, $U$ be a subset of E and $V$ a subset of $F$. A map $f:U\to V$ is said to be a Lipschitz diffeomorphism of degree $\gamma$ (a $\textrm{Lip}-\gamma$ diffeomorphism in short) if $f$ is $\textrm{Lip}-\gamma$ and bijective and $f^{-1}$ is also $\textrm{Lip}-\gamma$. We define in a similar way almost $\textrm{Lip}-\gamma$ diffeomorphisms.
	\end{Def}
	%%%% inverse function is a homeomorphism%%%%
	\begin{lemma}\label{InvFuncHomeomorphism} Let $\gamma, R >0$. Let $E$ be a Banach space, $x_0\in E$ and let $\varphi:B(x_0,R)\to E$ be a differentiable map such that $\mathrm{d}\varphi$ is $\textrm{Lip}-\gamma$. Assume that $\mathrm{d}\varphi(x_0)$ is invertible. Then, for every $M_1>0$ and $M_2>0$ such that $\|\mathrm{d}\varphi\|_{\textrm{Lip}-\gamma}\leq M_1$ and $\|\mathrm{d}\varphi(x_0)^{-1}\|\leq M_2$, there exists a constant $\delta$, depending only on $\gamma$ and $M_1 M_2$, such that for every $\alpha$ such that $\alpha< R$ and $0<\alpha\leq \delta$, the map $\varphi:B(x_0,\alpha)\cap \varphi^{-1}(V_0)\to V_0$, where $V_0=\varphi(x_0)+\mathrm{d}\varphi(x_0)(B(0,\alpha/2))$, is a homeomorphism, and we have $B(x_0,\alpha/3)\subseteq B(x_0,\alpha)\cap \varphi^{-1}(V_0)$. Moreover, $\varphi^{-1}$ is 1-H\"older with $\|\varphi^{-1}\|_{1}\leq 2M_2$.
	\end{lemma}
	\begin{proof}
	Let $M_1, M_2>0$ and assume that $\|\mathrm{d}\varphi\|_{\textrm{Lip}-\gamma}\leq M_1$ and $\|\mathrm{d}\varphi(x_0)^{-1}\|\leq M_2$. Let $\delta >0$  be a constant depending only on $\gamma$ and $M_1M_2$ such that the inequalities of lemma \ref{InjInequalities} hold true.
	%\varphi can be surjective
	Let $y\in V_0$. We will prove that there exists a unique point $x\in B(x_0,\alpha)$ such that $\varphi(x)=y$. Let $G$ be the map:
	\[\begin{array}{rccl}
	G:&\overline{B(x_0,\alpha)}&\longrightarrow&E\\
	&x&\longmapsto&x+\mathrm{d}\varphi(x_0)^{-1}(y-\varphi(x))\\
	\end{array}\]
	Note that $x\in \overline{B(x_0,\alpha)}$ is a fixed point of $G$ if and only if $\varphi(x)=y$. Let $x\in \overline{B(x_0,\alpha)}$, then by proposition \ref{InjInequalities}:
	\[\|G(x)-x_0\|\leq \|\mathrm{d}\varphi(x_0)^{-1}(y-\varphi(x_0))\|+ \|x-x_0+\mathrm{d}\varphi(x_0)^{-1}(\varphi(x)-\varphi(x_0))\|< \alpha\]
	Therefore $G(\overline{B(x_0,\alpha)})\subseteq\overline{B(x_0,\alpha)}$ and by lemma \ref{InjInequalities}, $G$ is a contraction. It has therefore a unique fixed point; denote it by $\tilde{x}$. Then we have:
	\[\|\tilde{x}-x_0\|=\|G(\tilde{x})-x_0\|< \alpha\]
	Hence $\tilde{x}\in B(x_0,\alpha)$, which proves the claim. Note that by the inequalities obtained in lemma \ref{InjInequalities}, we have $B(x_0,\alpha/3)\subseteq \varphi^{-1}(V_0)$.\\
	%\varphi is a homeomorphism
	We have shown that $\varphi: B(x_0,\alpha)\cap \varphi^{-1}(V_0)\to V_0$ is continuous and bijective. Therefore, $\varphi^{-1}: V_0 \to B(x_0,\alpha)\cap \varphi^{-1}(V_0)$ exists and is $1$-H\"older by the inequalities of lemma \ref{InjInequalities}, hence continuous. We conclude then that $\varphi$ is a homeomorphism (from $B(x_0,\alpha)\cap \varphi^{-1}(V_0)$ onto $V_0$).\\
	\end{proof}
	Before we proceed to the main theorem of this subsection, we prove a Lipschitz version of what is sometimes also labelled the inverse function theorem:
	%%%%
	%%%% INVERTIBILITY OF FUNCTIONS 
	%%%%
	\begin{lemma}\label{Invertibility} Let $0<\varepsilon \leq 1$. Let $E$ a Banach space and $U$ be an open subset of $E$. Let $\varphi:U\to E$ be an almost Lip-$(1+\varepsilon)$ map. Let $x_0\in U$ and assume that $\mathrm{d}\varphi(x_0)$ is invertible with continuous inverse. Then, there exists an open neighborhood $V$ of $x_0$ such that $\varphi:V\to \varphi(V)$ is a homeomorphism. Moreover $\varphi^{-1}$ is differentiable at $\varphi(x_0)$ and $\mathrm{d}\varphi^{-1}(\varphi(x_0))=(\mathrm{d}\varphi(x_0))^{-1}$.
	\end{lemma}
	\begin{proof}
	Without loss of generality, we assume that $\varphi$ is almost Lipschitz on domains of size $1$ of $U$. By the results of proposition \ref{InvFuncHomeomorphism}, there exists an open neighborhood $V$ of $x_0$ such that $\varphi:V\to \varphi(V)$ is a homeomorphism and that $\varphi^{-1}$ is 1-H\"older. We define the remainder-like map $S$ around $x_0$ for every $h\in E$ such that $\varphi(x_0)+h \in \varphi(V)$ by the following:
	\[S(h)=\varphi^{-1}(\varphi(x_0)+h)-x_0-(\mathrm{d}\varphi(x_0))^{-1}(h)\]
	To show that $\varphi^{-1}$ is differentiable at $x_0$ with the required derivative, it suffices then to show that $S(h)\underset{h\to 0}{=}o(\|h\|)$. Writing the Taylor expansion of $\varphi$ as an almost Lip-$(1+\varepsilon)$ map around $x_0$ with remainder $R$ of order 0, we get, for small enough $h$:
	\[\begin{array}{ccl}
	S(h)&=&(\mathrm{d}\varphi(x_0))^{-1}\left(\mathrm{d}\varphi(x_0)(\varphi^{-1}(\varphi(x_0)+h)-x_0)-h \right)\\
		&=&-(\mathrm{d}\varphi(x_0))^{-1}\left( R(\varphi^{-1}(\varphi(x_0)+h),x_0)\right)\\
	\end{array}
	\]
	By the definition of an almost Lipschitz map, we have for small enough $h$:
	\[\|R(\varphi^{-1}(\varphi(x_0)+h),x_0)\|\leq \|\varphi\|_{1,\mathrm{Lip}-(1+\varepsilon)}\|\varphi^{-1}(\varphi(x_0)+h)-x_0\|^{1+\varepsilon}\]
	and as $\varphi^{-1}$ is H\"older:
	\[\|\varphi^{-1}(\varphi(x_0)+h)-x_0\|\leq \|\varphi^{-1}\|_1\|h\|\]
	Therefore:
	\[\|S(h)\|\leq \|(\mathrm{d}\varphi(x_0))^{-1}\|\|\varphi\|_{1,\mathrm{Lip}-(1+\varepsilon)}\|\varphi^{-1}\|_1^{1+\varepsilon}\|h\|^{1+\varepsilon}\]
	which gives the sought statement.
	\end{proof}
	%%%% INVERSE FUNCTION THEOREM %%%%
	\begin{theo}[Inverse Function]\label{Invfunc}  Let $\gamma>1$ and $R >0$. Let $E$ be a Banach space, $x_0\in E$. Let $\varphi:B(x_0,R)\to E$ be an almost $\textrm{Lip}-\gamma$ map and assume that $\mathrm{d}\varphi(x_0)$ is invertible. Then, for every $M_1>0$ and $M_2>0$ such that $\|\mathrm{d}\varphi\|_{\textrm{Lip}-(\gamma-1)}\leq M_1$ and $\|\mathrm{d}\varphi(x_0)^{-1}\|\leq M_2$, there exists a constant $\delta$, depending only on $\gamma$ and $M_1 M_2$, such that for every $\alpha$ such that $\alpha< R$ and $0<\alpha\leq \delta$, the map $\varphi:B(x_0,\alpha)\cap \varphi^{-1}(V_0)\to V_0$, where $V_0=\varphi(x_0)+\mathrm{d}\varphi(x_0)(B(0,\alpha/2))$, is an almost Lip-$\gamma$ diffeomorphism, and we have $B(x_0,\alpha/3)\subseteq B(x_0,\alpha)\cap \varphi^{-1}(V_0)$. Moreover, the Lip-$(\gamma-1)$ norm of $\mathrm{d}\varphi^{-1}$ can be bounded from above by a constant depending only on $\gamma$ and $M_1$ and $M_2$.
	\end{theo}
	\begin{proof} Let $M_1, M_2>0$ and assume that $\|\mathrm{d}\varphi\|_{\textrm{Lip}-(\gamma-1)}\leq M_1$ and $\|\mathrm{d}\varphi(x_0)^{-1}\|\leq M_2$. Let $\delta >0$  be a constant depending only on $\gamma$ and $M_1M_2$ such that the inequalities of lemma \ref{InjInequalities} hold true and $\alpha>0$ such that lemma \ref{InvFuncHomeomorphism} holds true. Finally, define 
	\[V_0=\varphi(x_0)+\mathrm{d}\varphi(x_0)(B(0,\alpha/2))\]
	First note that lemma \ref{Invertibility} (together with the inequalities of lemma \ref{InjInequalities} ) shows that $\varphi^{-1}$ is differentiable at every point of $V_0$ and that for every $y\in V_0$, $\mathrm{d}\varphi^{-1}(y)=(\mathrm{d}\varphi(\varphi^{-1}(y)))^{-1}$. Let $n\in\mathbb{N}^*$ and $\varepsilon \in (0,1]$ such that $\gamma=n+\varepsilon$. We will show the assertion of the theorem by induction. More precisely, we will prove that for every $k\in[\![1,n]\!]$, $\varphi^{-1}$ is almost $\textrm{Lip}-(k+\varepsilon)$ and that there exists a constant $H_k$ depending only on $n$, $\varepsilon$, $M_1$ and $M_2$ such that $\|\mathrm{d}\varphi^{-1}\|_{\textrm{Lip}-(k+\varepsilon-1)}\leq H_k$. But let us first make some remarks:
	\begin{itemize}
	\item $V_0$ being open and convex, we can then use the criteria in proposition \ref{AlmostLipDiffMaps}  to show that $\varphi^{-1}$ is almost $\textrm{Lip}-(n+\varepsilon)$.
	\item If we denote by $i$ the inversion map on $\overline{B_{\mathcal{L}(E)}(\mathrm{d}\varphi(x_0),\frac{1}{2M_2})}$ (which is a smooth map and thus Lipschitz), $\mathrm{d}\varphi^{-1}$ can then be seen as the composition map of $\varphi^{-1}$, $\mathrm{d}\varphi$ and $i$:
	\[\mathrm{d}\varphi^{-1}: V_0\overset{\varphi^{-1}}\longrightarrow B(x_0,\alpha)\overset{\mathrm{d}\varphi}\longrightarrow \overline{B_{\mathcal{L}(E)}(\mathrm{d}\varphi(x_0),\frac{1}{2M_2})}\overset{i}\longrightarrow \mathcal{L}(E)\]
	For $\eta>0$, let $C_{\eta}$ denote the $\textrm{Lip}-\eta$ norm of $i$.
	\end{itemize}
	We start now our induction. Let $y,\tilde{y} \in V_0$. Since, $\mathrm{d}\varphi$ is $\varepsilon$-H\"older and $\varphi^{-1}$ is $1$-H\"older (propostion \ref{InvFuncHomeomorphism}), we have then:
	\[\begin{array}{lcl}
	\|\mathrm{d}\varphi^{-1}(y)-\mathrm{d}\varphi^{-1}(\tilde{y})\|&\leq&\|i(\mathrm{d}\varphi(\varphi^{-1}(y)))-i(\mathrm{d}\varphi(\varphi^{-1}(\tilde{y})))\|\\
	&\leq&C_1\|\mathrm{d}\varphi(\varphi^{-1}(y))-\mathrm{d}\varphi(\varphi^{-1}(\tilde{y}))\|\\
	&\leq&C_1\|\mathrm{d}\varphi\|_{\textrm{Lip}-\varepsilon}\|\varphi^{-1}(y)-\varphi^{-1}(\tilde{y})\|^{\varepsilon}\\
	&\leq&(2M_2)^{\varepsilon} C_1\|\mathrm{d}\varphi\|_{\textrm{Lip}-\varepsilon}\|y-\tilde{y}\|^{\varepsilon}\\
	\end{array}\]
	Hence, $\mathrm{d}\varphi^{-1}$ is $\varepsilon$-H\"older. Written as a composition map, we see that $\mathrm{d}\varphi^{-1}$ is bounded (by $C_1$). Consequently, $\varphi^{-1}$ is almost $\textrm{Lip}-(1+\varepsilon)$. Following theorem \ref{EmbedLip2}, let $m>0$ be constant dependent only on $n$ and $\varepsilon$, such that $\|\mathrm{d}\varphi\|_{\textrm{Lip}-\varepsilon}\leq m \|\mathrm{d}\varphi\|_{\textrm{Lip}-(n+\varepsilon-1)}$. Then:
	\[\|\mathrm{d}\varphi^{-1}\|_{\textrm{Lip}-\varepsilon}\leq H_1 \textrm{ , where } H_1=C_1\max(1,(2M_2)^{\varepsilon}m M_1)\]
	Let $k\in[\![1,n-1]\!]$. We assume that $\varphi^{-1}$ is almost $\textrm{Lip}-(k+\varepsilon)$ and that there exists a constant $H_k$ depending only on $n$, $\varepsilon$, $M_1$ and $M_2$ such that:
	\[\|\mathrm{d}\varphi^{-1}\|_{\textrm{Lip}-(k+\varepsilon-1)}\leq H_k\]
	Then, by proposition \ref{AlmostLipDiffMaps}, the almost Lipschitz semi-norm of $\varphi^{-1}$ on $V_0$ is bounded from above by $H_k$. As $\mathrm{d}\varphi^{-1}=i \circ \mathrm{d}\varphi \circ \varphi^{-1}$, then, by theorem \ref{CompolipfuncCASE22}, $\mathrm{d}\varphi^{-1}$ is $\textrm{Lip}-(k+\varepsilon)$ with a Lipschitz norm less than a constant $H_{k+1}$ depending only on:
	\begin{itemize}
	\item $k$ and $\varepsilon$, constants of the problem;
	\item $C_{k+\varepsilon}$ (which depends only on $k$, $\varepsilon$, $M_1$ and $M_2$);
	\item $\|\mathrm{d}\varphi\|_{\textrm{Lip}-(k+\varepsilon)}$ (which can be controlled using only $M_1$, $k$, $n$ and $\varepsilon$ by corollary \ref{EmbedLip2});
	\item $H_k$ which, by the induction hypothesis, depends only on $k$, $\varepsilon$, $M_1$ and $M_2$.
	\end{itemize}
	This ends the induction. Consequently, $\varphi:  B(x_0,\alpha)\cap \varphi^{-1}(V_0)\to V_0$ is an almost $\textrm{Lip}-(n+\varepsilon)$ diffeomorphism.\end{proof}
	\begin{Rem} In the context and notations of theorem \ref{Invfunc}, $\varphi$ is in fact a Lipschitz diffeomorphism. We use the notion of almost Lipschitzness to highlight that the most interesting attributes quantitatively depend only on the control of the Lipschitz norm of the derivative.
	\end{Rem}
	%%
	%The constant rank theorem
	%%
	\subsection{The constant rank theorem}

	%%%% LOCAL INVERSE %%%%
	\begin{Def}[Local Inverse] Let $E$ and $F$ be two sets. Let $U$ be a subset of $E$ and $\varphi:U\to F$ be a map. We say that an $E$-valued map $\hat{\varphi}$ defined on a subset of $F$ containing $\varphi(U)$ is a local inverse of $\varphi$ on $U$ if $\hat{\varphi}\circ\varphi_{|U}=\mathrm{Id}_U$.
	\end{Def}

	%%%% IMMERSIONS %%%%
	\begin{Def}[Immersions] Let $E$ and $F$ be two normed vector spaces. Let $U$ be a subset of $E$ and $\varphi:U\to F$ be a differentiable map. We say that $\varphi$ is an immersion if, for every $x\in U$, $\mathrm{d}\varphi(x)$ is injective.
	\end{Def}
	%%%% CONSTANT RANK FUNCTION THEOREM %%%%
	\begin{theo}[Constant Rank]\label{ConstantRankTheo} Let $\gamma>1$ and $(p,q,k)\in(\mathbb{N}^*)^2\times\mathbb{N}$ and $M_1$ and $M_2$ be two positive real numbers. Let $U$ be an open subset of $\mathbb{R}^p$ and
	\[\varphi=(\varphi_1,\ldots,\varphi_q):U\to \mathbb{R}^q\] 
	be an almost Lip-$\gamma$ map of rank at most $k$. Let $x_0 \in U$ such that $\mathrm{d}\varphi(x_0)$ is of rank $k$, and let $(i_1,\ldots,i_k)$ and $(j_1,\ldots,j_k)$ be, respectively, strictly ordered subsets of $[\![1,q]\!]$ and $[\![1,p]\!]$ such that $M=(\frac{\partial \varphi_{i_r}}{\partial x_{j_l}}(x_0))_{1\leq r,l\leq k}$ is invertible. We assume that:
		 \[\|\mathrm{d}\varphi\|_{\textrm{Lip}-(\gamma-1)}\leq M_1 \quad \textrm{and}\quad \|M^{-1}\|\leq M_2\] 
	Then, there exists a constant $c$, depending only on $\gamma$, $M_1$ and $M_2$ and:
	\begin{itemize}
	\item An almost $\textrm{Lip}-\gamma$ diffeomorphism $f:U_0\to f(U_0)$ defined on an open subset $U_0$ of $\mathbb{R}^p$ containing $x_0$ and such that $\max(\|\mathrm{d} f\|_{\textrm{Lip}-(\gamma-1)},\|\mathrm{d} f^{-1}\|_{\textrm{Lip}-(\gamma-1)})\leq c$.
	\item An almost $\textrm{Lip}-\gamma$ diffeomorphism $g:W\to g(W)$ defined on an open subset $W$ of $\mathbb{R}^q$ containing $\varphi(x_0)$ and such that $\max(\|\mathrm{d} g\|_{\textrm{Lip}-(\gamma-1)},\|\mathrm{d} g^{-1}\|_{\textrm{Lip}-(\gamma-1)})\leq c$.
	\end{itemize}
	such that, for all $(x_1,\ldots,x_p)\in f(U_0)$:
	\[g\circ \varphi \circ f^{-1}(x_1,\ldots,x_p)=(x_1,\ldots,x_k,0,\ldots,0)\]
	i.e.
		\[\begin{array}{ccc}
	 U_0(\subseteq \mathbb{R}^p)&\overset{\varphi}\longrightarrow&(\varphi(U_0)\subseteq)W(\subseteq \mathbb{R}^q)\\
	f\Big\downarrow&&\Big\downarrow g\\
	f(U_0)&\overset{\pi_{\mathbb{R}^k}}\longrightarrow&g(W)\\
	\end{array}\]
	If $k=p$, then there exists a constant $\delta$ depending only on $\gamma$, $M_1$ and $M_2$ such that for every $\alpha\in(0,\delta]$ satisfying $\overline{B(x_0,\alpha)}\subseteq U$, the above statement holds for $U_0=B(x_0,\frac{\alpha}{3\max(1,M_1M_2)})$, Moreover, $\varphi_{|U_0}$ is an injective immersion and admits a local inverse $\hat{\varphi}$ on $U_0$ that is almost $\textrm{Lip}-\gamma$ and such that:
	\[\|\mathrm{d}\hat{\varphi}\|_{\textrm{Lip}-\gamma}\leq c\]
	\end{theo}
	\begin{proof}%Change of variables
	We start first by two changes of variables that will enable us later to see $\varphi$ as a projection of the first $k$ variables around $x_0$. We will indentify $\mathbb{R}^p$ (resp. $\mathbb{R}^q$) with $\mathbb{R}^k\oplus\mathbb{R}^{p-k}$ (resp. $\mathbb{R}^k\oplus\mathbb{R}^{q-k}$). For $x\in \mathbb{R}^p$, we denote by $\overline{(x_{j_l})_{1\leq l\leq k}}$ the image of $x$ by the projection onto $\mathbb{R}^{p-k}$ which kernel is the span of $((e_{j_l})_{1\leq l\leq k})$, where $(e_i)_{1\leq i\leq p}$ is the canonical basis of $\mathbb{R}^p$. We define in a similar way the vector $\overline{(z_{i_r})_{1\leq r\leq k}}$ for $z\in \mathbb{R}^q$. Now, let $f_1$ and $g_1$ be the two following diffeomorphisms:
	\[\begin{array}{rccl}
	f_1:&\mathbb{R}^p&\to&\mathbb{R}^p\\
	&x&\mapsto&(((x-x_0)_{j_l})_{1\leq l\leq k},\overline{((x-x_0)_{j_l})_{1\leq l\leq k}})\\
	\end{array}\]
	and:
	\[\begin{array}{rccl}
	g_1:&\mathbb{R}^q&\to&\mathbb{R}^q\\
	&z&\mapsto&(((z-f(x_0))_{i_r})_{1\leq r\leq k},\overline{((z-f(x_0))_{i_r})_{1\leq r\leq k}})\\
	\end{array}\]
	Denote $U_1=f_1(U)$ and define $\tilde{\varphi}=(\tilde{\varphi}_1,\ldots,\tilde{\varphi}_q):=g_1\circ \varphi \circ {f_1}^{-1}$, $A=(\tilde{\varphi}_1,\ldots,\tilde{\varphi}_k)$ and $B=(\tilde{\varphi}_{k+1},\ldots,\tilde{\varphi}_q)$. By the identifications made above, we can write $\tilde{\varphi}$ as a map depending on two variables $(x_1,x_2)\in \mathbb{R}^k\oplus\mathbb{R}^{p-k}$ with values in the two-variable space $\mathbb{R}^k\oplus\mathbb{R}^{q-k}$:
	\[\tilde{\varphi}(x_1,x_2)=(A(x_1,x_2), B(x_1,x_2))\]
	and $\tilde{\varphi}$ is such that $\tilde{\varphi}(0,0)=(0,0)$. Trivially, $\mathrm{d}\tilde{\varphi}$, $\mathrm{d}A$ and $\mathrm{d}B$ are $\textrm{Lip}-(\gamma-1)$ and we have:
	\[\max(\|\mathrm{d}A\|_{\textrm{Lip}-(\gamma-1)},\|\mathrm{d}B\|_{\textrm{Lip}-(\gamma-1)})=\|\mathrm{d}\tilde{\varphi}\|_{\textrm{Lip}-(\gamma-1)}= \|\mathrm{d}\varphi\|_{\textrm{Lip}-(\gamma-1)}\]
	%Definition of the first chart
	Let $f_2$ be the map defined on $U_1$ by: 
	\[f_2(x_1,x_2)=(A(x_1,x_2),x_2)\]
	Then $f_2$ is differentiable at every point of $U_1$ and $\mathrm{d}f_2$ is $\textrm{Lip}-(\gamma-1)$ with:
	\[\|\mathrm{d}f_2\|_{\textrm{Lip}-(\gamma-1)}\leq \max(1,\|\mathrm{d}\varphi\|_{\textrm{Lip}-(\gamma-1)})\leq \max(1,M_1)\]
	The representation matrix of $\mathrm{d}f_2(0)$ in the canonical basis of $\mathbb{R}^p$ is under the form:
	\[\left(\begin{array}{cc}
	M&\tilde{M}\\
	0& \mathrm{I}_{p-k}\\
	\end{array}\right)\]
	where $\tilde{M}$ is some matrix in $\mathcal{M}_{k,p-k}$ that can be directly obtained from $\mathrm{d}\varphi(x_0)$. Hence $\mathrm{d}f_2(0)$ is invertible (its inverse can be explicitly given) and there exists a real number $C_{p,q}$ depending only on $p$ and $q$ such that:
	\[\|\mathrm{d}f_2(0)^{-1}\| \leq C_{p,q} \max(1,M_1) \max(1,M_2)\]
	Using theorem \ref{Invfunc}, let $\delta$ and $c$ be constants depending on $\gamma$, $M_1$ and $M_2$ such that, for every $\alpha \in (0,\delta]$ such that $\overline{B(0,\alpha)}\subseteq U_1$, the map: 
	\[f_2:B(0,\alpha)\cap f_2^{-1}(H)\to H\]
	is an almost $\textrm{Lip}-\gamma$ diffeomorphism, where:
	\[H=\textrm{d}f_2(0)(B(0,\alpha/2))\]
	and (we drop again the restriction signs $._{|}$):
	\[\|\mathrm{d} f_2^{-1}\|_{\textrm{Lip}-(\gamma-1)}\leq c\]
	As $f_2^{-1}$ is uniformly bounded on $H$ (by $\alpha$), then it is Lip-$\gamma$ with norm bounded by a constant depending only on $\gamma$, $M_1$ and $M_2$. Note also that $B$ is Lip-$\gamma$ on the bounded domain $B(0,\alpha)$ with norm upper-bounded by a constant depending only on $\alpha$, $\gamma$ and $M_1$. We also have $B(0,\alpha/3)\subseteq B(0,\alpha)\cap f_2^{-1}(H)$ (this remark will be of relevance in the case $k=p$).\\
	%Intermediate step
	We prove now that $\tilde{\varphi}\circ {f_2}^{-1}$ is independent of the second variable. Write ${f_2}^{-1}$ under the form: 
	\[{f_2}^{-1}(y_1,y_2)=(C(y_1,y_2),D(y_1,y_2))\]
	Then the identity $f_2\circ{f_2 }^{-1}=\mathrm{Id}$ yields:
	\[\forall (y_1,y_2)\in H:\quad D(y_1,y_2)=y_2\quad \textrm{and}\quad A(C(y_1,y_2),y_2)=y_1 \]
	$\tilde{\varphi}\circ {f_2}^{-1}$ is then given by:
	\[\forall (y_1,y_2)\in H:\quad \tilde{\varphi}\circ {f_2}^{-1}(y_1,y_2)=(y_1,B({f_2}^{-1}(y_1,y_2)))\]
	For lighter expressions, we define $\tilde{B}$ on $H$ by $\tilde{B}=B\circ{f_2}^{-1}$. As $\textrm{d}\tilde{\varphi}$ is of rank $k$ at most, $\textrm{d}(\tilde{\varphi}\circ {f_2}^{-1})$ is also at most of rank $k$ on $H$. For $(y_1,y_2)\in H$, the Jacobian matrix of $\tilde{\varphi}\circ {f_2}^{-1}$ at $(y_1,y_2)$ is under the form:
	\[\left(\begin{array}{cc}
	\mathrm{I}_{k}&0\\
	\frac{\partial \tilde{B}}{\partial y_1}(y_1,y_2)&\frac{\partial \tilde{B}}{\partial y_2} (y_1,y_2)\\
	\end{array}\right)\]
	As this matrix is of order $k$ at least, then necessarily $\frac{\partial \tilde{B}}{\partial y_2} (y_1,y_2)=0$. As $H$ is convex, we conclude that (on $H$) $\tilde{B}$ is independent of the second variable and so is $\tilde{\varphi}\circ {f_2}^{-1}$.
	Note that $B(0,\frac{\alpha}{2M_2})\subseteq \pi_{\mathbb{R}^k}(H)$ as for $a\in B(0,\frac{\alpha}{2M_2})$, we have $(a,0)\in H$. Following this remark, define $F$ on $B(0,\frac{\alpha}{2M_2})$ by $F(a)=\tilde{B}(a,0)$. We have then, for all $(y_1,y_2)\in H$:
	\[\tilde{\varphi}\circ {f_2}^{-1}(y_1,y_2)=(y_1,F(y_1))\]
	$F$ can be written as the composition $B\circ {f_2}^{-1} \circ i_{k,p}$, with $i_{k,p}(a)=(a,0)$ and is therefore, by theorem \ref{CompolipfuncCASE22}, Lip-$\gamma$ with norm bounded from above by a constant depending only on $\gamma$, $M_1$ and $M_2$. We define $U_2=H\cap B(0,\frac{\alpha}{2M_2})\oplus \mathbb{R}^{p-k}$, $\tilde{U}_1=f_2^{-1}(U_2)$, $U_0=f_1^{-1}(\tilde{U}_1)$ and $f=f_2\circ {f_1}_{\vert U_0}$. On the one hand, one easily sees that $f$ is Lip-$\gamma$ on $U_0$ with norm bounded from above by a constant depending only on $\gamma$, $M_1$ and $M_2$.On the other hand, $f^{-1}$ is almost Lip-$\gamma$ and $\|\mathrm{d} f^{-1}\|_{\textrm{Lip}-(\gamma-1)}=\|\mathrm{d} f_2^{-1}\|_{\textrm{Lip}-(\gamma-1)}$.
	
	%Definition of the second chart
	We end this proof by introducing a final diffeomorphism. Define the open set:
	\[W_1=B\left(0,\frac{\alpha}{2M_2}\right)\oplus \mathbb{R}^{q-k} =\left\{(z_1,z_2)\in \mathbb{R}^q|z_1\in B\left(0,\frac{\alpha}{2M_2}\right)\right\}\]
	Let $g_2$ be the map defined on $W_1$ by:
	\[g_2(z_1,z_2)=(z_1,z_2-F(z_1))\]
	$\mathrm{d}g_2$ is clearly $\textrm{Lip}-(\gamma-1)$ and $g_2$ is even an almost $\textrm{Lip}-\gamma$ diffeomorphism. The Lip-$(\gamma-1)$ norms of $\mathrm{d}g_2$  and $\mathrm{d}g_2^{-1}$ can be bounded from above by a constant depending only on $\gamma$, $M_1$ and $M_2$. For $(y_1,y_2)\in U_2$:
	\[g_2\circ\tilde{\varphi}\circ {f_2}^{-1}(y_1,y_2)=(y_1,0)\]
	Define $W=g_1^{-1}(W_1)$ and $g=g_2\circ {g_1}_{\vert W}$. Then $g$ is an almost $\textrm{Lip}-\gamma$ diffeomorphism and the Lip-$(\gamma-1)$ norms of $\mathrm{d}g$  and $\mathrm{d}g^{-1}$ can be bounded from above by a constant depending only on $\gamma$, $M_1$ and $M_2$.

	%Injectivity and Definition of the local inverse
	The previous argument simplifies in the case $k=p$. In this case, $U_2= B(0,\frac{\alpha}{2M_2})$. Then $g\circ\varphi\circ f^{-1}=i_{p,q}$, where 
	\[i_{p,q}:x\in \mathbb{R}^p \longmapsto (x,0) \in \mathbb{R}^p\oplus\mathbb{R}^{q-p}\]
	Let $\hat{\varphi}$ be the map $f^{-1}\circ\pi_{p,q}\circ g$ defined on $W$, where 
	\[\pi_{p,q}:(x,y)\in\mathbb{R}^p\oplus\mathbb{R}^{q-p}\longmapsto x\in \mathbb{R}^p\]
	is the projection on the first $p$ variables. Then $\hat{\varphi}$ is a local inverse of $\varphi$ on $U_0$ and a $\textrm{Lip}-\gamma$ map on $\varphi(U_0)$. Writing $\hat{\varphi}$ as the composition of the two maps $f^{-1}$ and $\pi_{p}\circ g$, one gets the control 	\[\|\mathrm{d}\hat{\varphi}\|_{\textrm{Lip}-(\gamma-1)}\leq C_{\gamma}\|\mathrm{d}f^{-1}\|_{\textrm{Lip}-(\gamma-1)}\]
	where $C_{\gamma}$ is a constant depending on $\gamma$. Taking $\beta= \frac{\alpha}{3\max(1,M_1M_2)}$, then we have:
	\[f_2(B(0,\beta)) \subseteq \mathrm{d}f_2(0)\left(B\left(0,\frac{\alpha}{2\max(1,M_1M_2)}\right)\right)
	\subseteq U_2\]
	which gives the required quantification of the neighborhood of $x_0$.
	\end{proof}

	%%%%%
	%% Bibliography
	%%%%%
	%\cleardoublepage
	%\addcontentsline{toc}{chapter}{Bibliography}
	
\end{document}